\documentclass[10pt,reqno,a4paper]{article} 

\usepackage{anysize}
\marginsize{3.0cm}{3.0cm}{2.2cm}{2.2cm}

\usepackage[english]{babel}
\usepackage[centertags]{amsmath}
\usepackage{amsfonts,amssymb,amsthm}
\usepackage[svgnames]{xcolor}
\usepackage{comment}
\usepackage{hyperref} 
\usepackage{mathrsfs}
\usepackage[sans]{dsfont}
\usepackage{accents}

\let\OLDthebibliography\thebibliography
\renewcommand\thebibliography[1]{
  \OLDthebibliography{#1}
  \setlength{\parskip}{0pt}
  \setlength{\itemsep}{0pt plus 0.3ex} }

\numberwithin{equation}{section}

\theoremstyle{plain}
\newtheorem{theorem}{Theorem}[section]

\newtheorem{lemma}[theorem]{Lemma}

\newtheorem{prop}[theorem]{Proposition}

\newtheorem{ass}[theorem]{Assumptions}
\newtheorem{corollary}[theorem]{Corollary}
\newtheorem{definition}[theorem]{Definition}

\theoremstyle{definition}
\newtheorem{remarkbase}[theorem]{Remark}
\newenvironment{remark}{\pushQED{\qed} \remarkbase}{\popQED\endremarkbase}

\newcommand{\N}{{\mathbb N}}
\newcommand{\R}{{\mathbb R}}

\newcommand{\mB}{\mathcal{B}}

\newcommand{\mD}{\mathcal{D}}
\newcommand{\mE}{\mathcal{E}}
\newcommand{\mF}{\mathcal{F}}
\newcommand{\mG}{\mathcal{G}}

\newcommand{\mM}{\mathcal{M}}
\newcommand{\mN}{\mathcal{N}}

\newcommand{\mS}{\mathcal{S}}
\newcommand{\mT}{\mathcal{T}}
\newcommand{\mU}{\mathcal{U}}

\newcommand{\mW}{\mathcal{W}}

\newcommand{\g}{\gamma}

\newcommand{\e}{\varepsilon}
\newcommand{\ph}{\varphi}

\newcommand{\Om}{\Omega}
\newcommand{\om}{\omega}

\newcommand{\s}{\sigma}
\renewcommand{\t}{\tau}
\renewcommand{\th}{\vartheta}

\newcommand{\la}{\langle}
\newcommand{\ra}{\rangle}
\newcommand{\pa}{\partial}
\renewcommand{\div}{\mathrm{div}\,}
\newcommand{\grad}{\nabla}

\title{Optimal control of mean-field limit of multiagent systems with and without common noise}
 
\author{\normalsize{Giuseppe La Scala}}
\date{} 

\begin{document}

\maketitle

\noindent
\textbf{Abstract.} We consider a generic, suitable class of optimal control problems under a constraint given by a finite-dimensional SDE-ODE system, describing a system of two interacting species of particles: the \emph{herd}, described by SDEs, and the \emph{herders}, described by ODEs with the addition of a control function. 

In particular, we firstly show that for a low number of herders and for the limit of large number of herd individuals, the SDE-ODE system can be approximated by an infinite-dimensional system given by a McKean-Vlasov single SDE coupled with ODEs. Then, thanks to this we show the $\Gamma-$convergence of the optimal control problem for the finite-dimensional system to a certain optimal control problem for the mean-field system.

Differently from \cite{ACS}, we do not consider an additive noise for the herd, but a more general class, given by \emph{idiosyncratic} noises (due to a single herd individual) together with \emph{common} noise (due to how the environment affects the whole herd), and they are independent one from another. As well as this, we consider a more general class of control functions in the ODEs for herders, where the control is applied not only on the herd dynamics, but also on the herd one.

\bigskip


\bigskip

\tableofcontents

\bigskip

\section{Introduction}

\emph{Model and motivations.} 
In the present paper, we consider the following SDE-ODE system:
\begin{align}&dX_n^{(N)}(t)=\Big\{\frac{1}{N}\sum_{i=1}^N H_1(X_n^{(N)}(t) - X_i^{(N)}(t)) + \frac{1}{M}\sum_{j=1}^M K_1(Y_j^{(N)}(t) - X_n^{(N)}(t))\Big\}dt \notag
\\&\qquad\qquad+ \Big\{\s_*^{(i)}(t, Y^{(N)}(t), X_n^{(N)}(t),\mu_N(t))\Big\}dW_n^{(i)}(t) \notag
\\&\qquad\qquad+ \Big\{\s_*^{(c)}(t,Y^{(N)}(t),X_n^{(N)}(t),\mu_N(t))\Big\}dW^{(c)}(t), \label{SDE.finite}
\\&\notag
\\&\frac{d}{dt}Y_m^{(N)}(t)=\frac{1}{N}\sum_{i=1}^N K_2(Y_m^{(N)}(t) - X_i^{(N)}(t)) + \frac{1}{M}\sum_{j=1}^M H_2(Y_m^{(N)}(t)-Y_j^{(N)}(t)) \notag
\\&\qquad\qquad\quad+ u_m(t,Y^{(N)}(t),\mu_N(t)),\label{ODE.finite}
\\&\notag
\\&X_n^{(N)}(0)=X_n^{(0)},\qquad Y_m^{(N)}(0)=Y_m^{(0)}.\qquad1\le n\le N,\quad1\le m\le M.\label{Initial.data.finite}
\end{align}
Let us comment this model. It describes the motion of two interacting species moving in $\R^d$, with $d\in\{1,2,3\}$. 

The first one is what we call the \emph{herd}, made of $N$ individuals labelled by $X_n^{(N)}$, and the second one is what we call the \emph{herders}, made of $M$ individuals labelled by $Y_m^{(N)}$; we have also introduced the notations
\begin{align}&X^{(N)}:=(X_n^{(N)})_{1\le n\le N},\qquad Y^{(N)}:=(Y_m^{(N)})_{1\le m\le M},
\\&X^{(N,0)}:=(X_n^{(0)})_{1\le n\le N},\qquad Y^{(N,0)}:=(Y_m^{(0)})_{1\le m\le M},
\\&\mu_N(t):=\frac{1}{N}\sum_{j=1}^N\delta_{X_j^{(N)}(t)},\label{convenient.notation}
\end{align}
where $\mu_N$ is the statistical distribution of the herd, or else, its \emph{empirical measure}.
They are particles whose motion is due to the following factors:
\begin{itemize}

    \item the pairwise interaction between herd individuals is described by a kernel $H_1$, while the one between herders is described by $H_2$. Collisions between two individuals of the same species are allowed, and these interactions hold a moderate intensity (see \eqref{K.L.Lipschitz});
    
    \item the pairwise (not symmetric, in general) interaction between herd individuals with each herder is described by $K_1$, while the one between herders with each herd individual is described by $K_2$. Collisions between two individuals of different species are allowed, and these interactions hold a moderate intensity (see again \eqref{K.L.Lipschitz});

    \item the motion of herd individuals is also induced by two further \emph{stochastic} independent effects: the \emph{idiosyncratic noises} $\s_*^{(i)}dW_n^{(i)}$, due to the single herd individuals, and the \emph{common noise} $\s_*^{(c)}dW^{(c)}$, due to how the whole herd interacts with the outer environment. They can generally depend on time as well as on the herders positions, the herd individual own position and the statistical herd position, and their intensities are supposed to be moderate (see \eqref{idiosyncratic.noise.is.continuous}-\eqref{idiosyncratic.noise.is.lipschitz} and \eqref{common.noise.is.continuous}-\eqref{common.noise.is.lipschitz});

    \item the motion of herders is also induced by a \emph{stochastic control function} $u_m$, different from herder to herder and depending on time as well as their own position and that of the whole herd $\mu_N$. Each control function is considered to be a continuous process adapted to the same filtration of the common noise, to have a moderate intensity (see \eqref{u.lipschitz.in.measure}), and it is chosen to be (compare Assumptions \ref{assumptions.control})
    \begin{equation}\label{separation of variable}u_m(t,
    y,\nu):=h_m(t)g_m(y,\nu).
    \end{equation}
    
\end{itemize}

In a number of applications, there is an interest in finding minima or maxima to functionals of the kind
\begin{align}\label{F.N.finite} \mF_N(h,g):&=\mathbb{E}\Big[\int_0^T\Psi_\rho(h(t),g(Y^{(N)}(t),\mu_N(t)))\,dt\Big]\notag
\\&+\mathbb{E}\Big[\int_0^T\Psi_\t(t,Y^{(N)}(t),\mu_N(t))\,dt\Big] \notag
\\&+\mathbb{E}\Big[\Psi_\e(Y^{(N)}(T),\mu_N(T))\Big].
\end{align}
subjected to the constraint \eqref{SDE.finite}-\eqref{Initial.data.finite}.
To explain each term, the first integral is what it is usually called the \emph{running cost}, while the last term is usually called the \emph{endpoint cost}. The term in the middle, as well as the endpoint cost, does not depend on the control functions $h=(h_1,\dots,h_M)$ and $g=(g_1,\dots,g_M)$ but purely on dynamical variables, and differently from the endpoint cost, it takes into account all of the dynamics of \eqref{SDE.finite}-\eqref{Initial.data.finite}: we then call it \emph{transient cost}. This kind of models has been taken into consideration in \cite{ABCK, AFRB, LRMA, Pierson.Schwager, Sebastian.Montijano}.

\bigskip

In many models like those studied in the previous papers, the herd is made of a huge number $N$ of individuals: even though from a theoretical point of view the SDE-ODE system \eqref{SDE.finite}-\eqref{Initial.data.finite} and the optimal control problem for \eqref{F.N.finite} are well-posed under some suitable regularity assumptions (see Assumptions \ref{assumptions}, Lemma \ref{lemma:finite.well.posedness} for the equations and Assumptions \ref{assumptions.control}, Lemma \ref{lemma:F.N.has.solution} for the optimal control problem), from a numerical one simulating solutions and finding optimal controls could be a very hard task to achieve. As a result, one would be then interested in finding an \emph{effective model} for both the system \eqref{SDE.finite}-\eqref{Initial.data.finite} and the functional \eqref{F.N.finite}.

\bigskip

\emph{Heuristics: the mean-field limit without and with common noise.}
The natural way to produce such an effective model is to send $N\to\infty$ and see what occurs; let us start by the constraints \eqref{SDE.finite}-\eqref{Initial.data.finite}.

Heuristically speaking, since the interaction terms with the herd are scaled by a factor $O(N^{-1})$, they can be seen as 
$H_1*\mu_N$ and $K_2*\mu_N$, while the interaction terms with the herders are not scaled by factors depending on $N$, so it does not hold the same thing.

The choice of that scaling factor is not by chance. One would have that the eventual convergence of the statistical distributions $\mu_N$ for the herd to a limit distribution $\mu$ for the "infinite" herd would imply that as well, $H_1*\mu_N\sim H_1*\mu$ and $K_2*\mu_N\sim K_2*\mu$. Then each individual $X_n$ of the "infinite" herd and each herder $Y_m$ is expected to satisfy the \emph{mean-field problem}
\begin{align}&dX_n(t)=\Big\{H_1*\mu^{(i)}(t)(X_n(t)) + \frac{1}{M}\sum_{j=1}^M K_1(Y_j(t) - X_n(t))\Big\}dt \notag
\\&\qquad\quad+\Big\{\s_*^{(i)}(t,Y(t),X_n(t),\mu^{(i)}(t))\Big\}\,dW_n^{(i)}(t) + \Big\{\s_*^{(c)}(t,Y(t),X_n(t),\mu^{(i)}(t))\Big\}\,dW^{(c)}(t), \label{SDE.mean.field}
\\&\notag
\\&\frac{d}{dt}Y_m(t)=K_2*\mu^{(i)}(t)(Y_m(t)) + \frac{1}{M}\sum_{j=1}^M H_2(Y_m(t)-Y_j(t)) + u_m(t,Y(t),\mu^{(i)}(t))\label{ODE.mean.field},
\\&\notag
\\&X_n(0)=X_n^{(0)},\qquad Y_m(0)=Y_m^{(0)},\qquad\qquad n\in\N,\quad1\le m\le M \label{Initial.data.mean.field}.
\end{align}
When there is no common noise, that is $\s_*^{(c)}\equiv0$, this happens when the herd individuals are mutually independent: in this case, at the mean-field limit they remain independent because heuristically, the equations \eqref{SDE.mean.field} are uncoupled, and then any observables along them become uncorrelated. This phenomenon is called \emph{propagation of chaos}. 

Now, one wonders how to get the limit measure $\mu^{(i)}$. This is basically based on $\hat{\text{I}}$to calculus applied to \eqref{SDE.finite}-\eqref{Initial.data.finite}: if one considers any observable $\ph$ satisfying some suitable regularity properties then one would get something of the kind
\begin{align*}\la\mu_N(t),\ph\ra-\la\mu_N(0),\ph\ra
= \int_0^t\la\mu_N(s),\grad\ph\cdot V(s,\mu_N(s),\cdot) + \frac12\text{tr}\{H(\ph)\s_*(s,\mu_N(s),\cdot)
 \ra\,ds + O(N^{-\frac12})     
\end{align*}
where $V$ is defined in \eqref{V.def}. Then, at the limit $\mu^{(i)}$ is expected to be a measure-valued solution (uniquely, indeed) for the Fokker-Planck PDE Cauchy Problem
\begin{align}&\pa_t\mu^{(i)} - \frac12\text{tr}(H(\s_*\mu^{(i)}))=-\div(V\mu^{(i)}),    \label{Deterministic.Fokker.Planck.PDE}
\\&\mu^{(i)}(0)=\text{Law}(X_n^{(0)});  \label{Deterministic.Fokker.Planck.Initial.datum}
\end{align}
in other words, $\mu^{(i)}$ would correspond to the law of each mean-field process $X_n$.

Instead, when there is common noise, that is $\s_*^{(c)}\not\equiv0$, the mean-field SDEs are still coupled by the presence of the common noise, and so properly speaking we do not expect propagation of chaos even though initially, the herd individuals are independent one from another. To fix this problem and then get the mean-field limit, the following observation is crucial: if we know how the common noise is realized, then it becomes like a known additive term into the equations \eqref{SDE.mean.field}-\eqref{Initial.data.mean.field}, and so we would get uncoupled equations. Technically speaking, the idea is to condition the discrete problem \eqref{SDE.finite}-\eqref{Initial.data.finite} with respect to the filtration over which the common noise is measurable, and then take the mean-field limit. In this case, it turns out that we have the onset of \emph{conditional propagation of chaos}; precisely, in the mean-field limit the herd individuals become independent and observables along them uncorrelated, given the common noise.
As a result, the limit measure $\mu^{(i)}$ of each herd individual given the common noise is a continuous process being adapted to the same filtration of the common noise, and it corresponds to the conditional law of the corresponding $X_n$ given the common noise. Thus, it is almost surely a measure-valued solution of the Fokker-Planck SPDE Problem 
\begin{align}&d\mu^{(i)}=\Big\{\frac12\text{tr}(H(\s_*\mu^{(i)}))-\div(V\mu^{(i)})\Big\}dt - \div((\s_*^{(c)}dW^{(c)})\mu^{(i)})\label{parabolic.equation.common}
\\&\mu^{(i)}(0)=\text{Law}(X^{(0)}), \label{parabolic.initial.data.common}
\end{align}
where $\s_*^2$ is defined in \eqref{sigma.star.def}. This must be meant in the sense of \eqref{duality.fokker.planck.common}.

Ultimately, we expect convergence for the functional $\mF_N$ to a functional of the kind
\begin{align}\mF(h,g):&=\mathbb{E}^{(c)}\Big[\int_0^T\Psi_\rho(h(t),g(Y(t),\mu(t)))\,dt\Big]  \notag
\\&+\mathbb{E}^{(c)}\Big[\int_0^T\Psi_\t(t,Y(t),\mu(t))\,dt\Big] \notag
\\&+\mathbb{E}^{(c)}[\Psi_\e(Y(T),\mu(T))]\label{F.mean.field}
\end{align}
defined on the constraints \eqref{SDE.mean.field}-\eqref{Initial.data.mean.field}; the expectation is done with respect to the common noise probability space, which disappears if there is no common noise. 

\bigskip

\emph{Purpose and structure of the paper, and strategy of proof.}
In this paper, we want to establish the existence of minima for $\mF_N$ under the constraint \eqref{SDE.finite}-\eqref{Initial.data.finite} (Lemma \ref{lemma:F.N.has.solution}), and for $\mF$ under \eqref{SDE.mean.field}-\eqref{Initial.data.mean.field} (Lemma \ref{lemma:F.has.solution}); then, we want to prove that minima for $\mF_N$ $\Gamma-$converge to those of $\mF$ (Theorem \ref{thm:Gamma.convergence}). This will be done in Section 6 under regularity assumptions on controls which generalize those provided in \cite{ACS}.

To this term, we firstly need to show that the discrete constraints \eqref{SDE.finite}-\eqref{Initial.data.finite} converge to the mean-field ones \eqref{SDE.mean.field}-\eqref{Initial.data.mean.field} (Theorem \ref{thm:propagation.of.chaos}), which is an outstanding result alone; this is done between Section 3, where well-posedness of strong solutions is established, and Section 4, where the mean-field limit is established. We will follow the approach of Sznitman \cite{Sznitman}, consisting in Gr\"onwall-based estimates
and the Law of Large Numbers applied to the mean-field distribution of herd individuals (which is available thanks to independence hypothesis on the initial data), together with its extension provided in \cite{CDL} for the common noise case. From this, we will derive also the propagation of chaos for zero-common noise case, and the conditional one for the general one (Corollary \ref{cor:propagation.of.chaos}). Moreover, in Section 5 we will show that the mean-field limit (conditional) measure satisfies the stochastic Fokker-Planck equation \eqref{parabolic.equation.common}-\eqref{parabolic.initial.data.common}, and in the zero-common noise case, the PDE \eqref{Deterministic.Fokker.Planck.PDE}-\eqref{Deterministic.Fokker.Planck.Initial.datum}; in the latter case, we will prove uniqueness of its solution by standard duality argument based on the Feynman-Kac Formula (Lemma \ref{lemma:Feynman.Kac}). 

\bigskip

\emph{Differences with other related works, and technical issues.} The problem described above has been already studied in \cite{ACS} in the case where there is only one additive idiosyncratic stochastic noise, the control functions $u_m$ in \eqref{ODE.finite}, \eqref{ODE.mean.field} depend only on time and on the statistical distribution of the herd, and the cost functionals \eqref{F.N.finite},\eqref{F.mean.field} do not contain also the endpoint cost. In particular, we mean to provide a more general framework for optimal control problems with SDE-ODE constraints, and their mean-field limits. 

\medskip

Our approach relies on the fact that the techniques introduced in \cite{ACS} can be extended to the present work. However, differently from the additive idiosyncratic noise case, the multiplicative noises appear when trying to get Gr\"onwall-based estimates, and they \emph{do not} cancel out: to fix this, we introduce the Burkholder-Davis-Gundy Inequality (see \eqref{Doob.maximal.inequality}) to control all the moments of the noises and therefore, to get the convergence estimates \eqref{propagation.of.chaos}-\eqref{convergence.of.measures} and then \eqref{propagation.of.chaos.1}-\eqref{convergence.of.measures.1}, relying also on the Quantitative Law of Large Numbers \eqref{law.of.large.numbers} proved in \cite{Fournier.Guillin}. In \cite{ACS}, only the $1-$st moment was estimated, which is in fact enough to have the mean-field limit and the propagation of chaos. We point out that differently from \cite{ACS}, where Doob Maximal Inequality has been used, we \emph{necessarily} need the Burkholder-Davis-Gundy Inequalities for the present setting to get the estimates \eqref{propagation.of.chaos}-\eqref{convergence.of.measures}: indeed, when the initial data are only square-integrable ($p=2$), we cannot get estimates for $q<p$ by using only the former inequality.

\medskip

Another difference is about the fact that to get uniqueness of the measure-valued solution for the Fokker-Planck equation in the zero-common noise case, we have not relied on the requirement of finite entropy as in \cite{ACS}, but on the weaker classical duality argument based on the solvability of the Backward Kolmogorov problem \eqref{Parabolic.problem}-\eqref{Parabolic.initial.datum} through the Feynman-Kac Formula \eqref{feynman.kac}, which is available since the coefficients of the equations are regular enough (see Assumptions \ref{assumptions}). As for the case $\s_*^{(c)}\not\equiv0$, for our knowledge the uniqueness of measure-valued solutions for \eqref{parabolic.equation.common}-\eqref{parabolic.initial.data.common} is a delicate issue and it holds at least when the coefficients have further differential properties than those we provided in Assumptions \ref{assumptions}, see for instance \cite{Kolokoltsov.Troeva, DMT}.

\medskip

As far as controls are concerned, in \cite{AACS} it has been considered a more general optimal control problem in the case of additive noise, where no separation of variables like in \eqref{separation of variable} has been made. Since in applications there is freedom in choosing controls, in the present paper we decided to consider separation of time control from that of dynamical variables, since we want to shift any possibility of generality on dynamical variables $\mu_N,\,Y^{(N)}$ for the discrete problem, and $\mu,\,Y$ for the mean-field one.

\bigskip

\emph{Open problems and future perspectives.} In the actual paper, we basically treat the well-posedness of the optimal control problems \eqref{discrete.optimal.control.problem}, \eqref{mean.field.optimal.control.problem} and the $\Gamma-$convergence of the first to the second. However, in applications the computation and the simulation of solutions to such problems are required: to this purpose, one should get a Pontryagin Maximum Principle or Dynamic Programming Principle, which could be subject to further investigations.

Another issue is that we have supposed in the constraints that the interaction coefficients have a moderate intensity. This allows us to follow Sznitman approach for getting quantitative convergence estimates \eqref{propagation.of.chaos} and \eqref{convergence.of.measures}. When it comes to singular interactions, this cannot be given for granted, and very often the latter appear in models, like Coulomb interactions. This would be another possible direction of future developments.

\bigskip

\emph{Related Literature.} 
The problem of herding, which the actual paper mainly concerns about, has been considered in the previously cited papers \cite{ABCK, AFRB, LBSRA, LRMA, Pierson.Schwager, Sebastian.Montijano} for applications in robotics, and it was theoretically studied \cite{ACS} in the case of additive idiosyncratic noise with control only on the herd distribution.

\medskip

The problem of mean-field approximations of multi-particle systems dates back to the papers \cite{Kac}, where, inspired by the cornerstone paper \cite{Boltzmann}, the notion of propagation of chaos was introduced in kinetic models. The extension to SDE and diffusion models was done in \cite{McKean}. As for quantitative results for propagation of chaos when there are moderate interactions in SDE and diffusion models, we refer to \cite{Oelschlager, Sznitman, Meleard} and to the review \cite{Jabin.Wang}. All these works on mean-field limit are related to the case of zero-common noise; in the presence of common noise, we refer to \cite{Coghi.Flandoli, CDFM} for the conditional propagation of chaos, while in \cite{DMT, Kolokoltsov.Troeva} the well-posedness of the Stochastic Fokker-Planck equation is studied. We further consider the papers \cite{AACS.1, BCC, DIRT, FPZ, GLM, Jabin.Wang.1, LLY} when the drift coefficients are singular.

\medskip

For optimal control mean-field problems through $\Gamma-$limit approach, we refer for instance to \cite{AAMS, ACFK, Fornasier.Solombrino, AACS}, while for Pontryagin maximum principle in Wasserstein spaces for such problems we refer to \cite{BFRS, Bonnet, Bonnet.Frankowska, Bonnet.Rossi, BPTTR, CPV, FPR}. In the case of common noise, we refer to \cite{CDL, Carmona.Delarue}, where in particular it was given a general framework for mean-field games with and without common noise.

\medskip

For further applications of mean-field approximations and control, we refer to \cite{Keller.Segel, CDFSTB, LLN, MBHJRB} for Biology and Chemistry, and to \cite{CHDB, Cucker.Smale, Donofrio.Hernandez, DMPW, Leonard.Fiorelli, PGE, Toscani} for multi-agent systems and opinion models. 

\medskip

For general treatment of probability calculus and stochastic processes, we refer to \cite{Baldi, Ikeda.Watanabe, Oksendal, Pavliotis, Revuz.Yor}; for a treatment of Wasserstein spaces and Quantitative Law of Large Numbers, we refer to \cite{ABS, Carmona.Delarue, Fournier.Guillin, Villani}.

\bigskip

\section{Preliminaries}

In the first two subsections, we introduce some notations, known concepts and results that will be used in the next sections, while in the third one we will build assumptions and definitions necessary to prove our results.

\subsection{Random variables and stochastic calculus}

Let $(B,d_B)$ be a complete metric space and for any fixed horizon $T>0$, let $C([0,T];B)$ be the set of continuous functions defined over $[0,T]$ and taking values in $B$. 

For any given set $\Om$ and any complete $\s-$algebra $\mF$ over $\Om$, we denote by $\mM(\Om;B)$ the set of random variables defined on $\Om$ which are $\mF-$measurable and take values over $B$. We denote by $\mu:=\text{Law}(X)$ the law of $X$, which is such that for all $A\in B$,
\begin{equation*}\mu(A):=\mathbb{P}(X\in A)=\mathbb{P}(\{\om\in\Om\colon\,X(\om)\in A\}),
\end{equation*}
and for all $\mu-$integrable functions $\ph\colon\,B\longmapsto\R$ we will write
\begin{equation*}\la\mu,\ph\ra:=\int_B\ph(x)\,\mu(dx).
\end{equation*}
We denote also by $\mathbb{E}\colon\,\mM(\Om;B)\longmapsto\R$ the associated expectation operator; $X$ is said to be centered if $\mathbb{E}[X]=0$. Two random variables $X_1\in\mM(\Om;B_1),\,X_2\in\mM(\Om;B_2)$ are said to be independent if for any $A_1\in B_1,\,A_2\in B_2$, one has
\begin{equation*}\mathbb{P}(X_1\in A_1,\,X_2\in A_2)=\mathbb{P}(X_1\in A_1)\cdot\mathbb{P}(X_2\in A_2).
\end{equation*}
If for some $d\in\N$ we have $B=\R^d$ and $X\in\mM(\Om;\R^d)$, given a sub-$\s-$algebra $\mD\subset\mF$, the conditional expectation $\mathbb{E}[X|\mD]$ is defined to be the equivalence class of random variables $Z$ which are $\mD-$measurable such that for any $D\in\mD$,
\begin{equation*}\int_D Z\,d\mathbb{P}=\int_D X\,d\mathbb{P};
\end{equation*}
we recall that its existence and uniqueness is guaranteed by Radon-Nikodym Theorem.

About the conditional expectation, we have the so called Freezing Lemma:
\begin{lemma}[Lemma 4.1, \cite{Baldi}]\label{lemma:freezing}
Let $(\Om,\mF,\mathbb{P})$ be a probability space, let $\mG,\mD$ two independent sub-$\s-$algebras of $\mF$. Let $X$ be a $\mD-$measurable random variable taking values in the measurable space $(E,\mE)$, let $\Psi\colon\,E\times\Om\longmapsto\R$ be an $\mE\otimes\mG-$measurable function such that $\om\longmapsto\Psi(X(\om),\om)$ is integrable.

Then, for all $\xi\in\Om$ one has that $f(\xi):=\mathbb{E}[\Psi(X(\cdot),\cdot)|\mD](\xi)$ is $\mD-$measurable and it holds the identity
\begin{equation}f(\xi)=\mathbb{E}[\Psi(X(\xi),\cdot)].
\end{equation}

\end{lemma}

If $B=\R^d$ for some $d\in\N$, we call characteristic function of $X$ the function 
\begin{equation*}\hat\mu(\th):=\mathbb{E}[e^{i\la\th,X\ra}]=\int_{\R^d}e^{i\la\th,x\ra}\mu(dx),\qquad\forall\th\in\R^d,\end{equation*}
where we are temporarily denoting by $\la\cdot,\cdot\ra$ the Euclidean scalar product on $\R^d$.
If in particular there exist a vector $a\in\R^d$ (the average) and a matrix $A\in\R^{d\times d}$ (the covariance matrix) such that $\hat\mu(\th)=e^{i\la\th,a\ra}\cdot e^{\la AA^*\th,\th\ra}$ (where $A^*$ is denoted to be the transposed of $A$), then $X$ is said to be a Gaussian random variable, and we say that $X\sim\mN(a,AA^*)$; if $a=0$, $X$ is said to be centered. A vector of random variables $(X_n)_{1\le n\le N}$ is said to be Gaussian if for any $\alpha_1,\dots,\alpha_N\in\R$, the sum $\sum_{n=1}^N\alpha_n X_n$ is a Gaussian random variable.

We point out the following important fact:
\begin{prop}\label{sum.of.Gaussians} Given two Gaussian independent random variables $X_1\sim\mN(a_1,A_1),\,X_2\sim\mN(a_2,A_2)$, we have that $\alpha_1 X_1+\alpha_2 X_2\sim\mN(\alpha_1 a_1+\alpha_2 a_2,\,\alpha_1^2A_1A_1^* + \alpha_2^2A_2A_2^*)$ 
\end{prop}

\medskip

Now, if one considers any measurable random function $X\colon\,\Om\times[0,T]\longmapsto B$, then $X$ is said to be a stochastic process, and it is said to be continuous if for all $\om\in\Om$, $X(\om,\cdot)\in C([0,T];B)$. More, if one defines a filtration $(\Om,\mF=\mF_0,(\mF_t)_{t\in[0,T]}, \mathbb{P})$ and for all $t\in[0,T]$ one has that $X(t)$ is $\mF_t-$measurable, we say that the stochastic process $X$ is $(\mF_t)_{t}-$adapted. 

\medskip

If $B=\R^d$ for any $d\in\N$ and $X$ is a stochastic process such that for all $t\in[0,T]$, $X(\cdot,t)$ is a Gaussian random variable, then $X$ is said to be a Gaussian process if for any $N\in\N$ and times $t_1,\dots,t_N\in[0,T]$, one has that the vector of random variables $(X(t_n))_{1\le n\le N}$ is Gaussian. If such a process is continuous, then the mean function $a(t):=\mathbb{E}[X(t)]$ and the covariance function $R(t,s):=\mathbb{E}[(X(t)-a(t))(X(s)-a(s))]$ are still continuous, and the latter is also semipositive definite. In particular, if $a(t)=0$ and $R(t,s)=\min\{t,s\}$ for all $t,s\in[0,T]$, we say that $X$ is said to be a Brownian motion; in this case, we will refer to it with the notation $W$.

\medskip

A $(\mF_t)_t-$adapted process $M$ with $B=\R^d$ is said to be a martingale if $M_t$ is integrable for all $t\in[0,T]$ and $\mathbb{E}[M_t|\mF_s]=M_s$ for all $0\le s\le t \le T$. Any Brownian motion is an example of martingale.

Thanks to them, one can build $\hat{\text{I}}$to stochastic integrals with respect to martingales with finite quadratic variation as follows. We call elementary process some of the kind $X(t):=\sum_{n=0}^{N-1} X_i\cdot 1_{[t_n,t_{n+1})}$, where $\{t_i\colon\,0\le i\le N\}$ is a partition of $[0,T]$ and $1_{[t_n,t_{n+1})}$ is the indicator function over $[t_n,t_{n+1})$. Its $\hat{\text{I}}$to stochastic integral with respect to a martingale $M$ is defined to be
\begin{equation*}\int_0^T X(s)\cdot dM(s):=\sum_{n=0}^{N-1} X_n\cdot[M(t_{n+1}) - M(t_n)].
\end{equation*}
By the martingale property of $M$, one gets that if such elementary process $X$ is centered, then
\begin{equation}\label{ito.isometry}\mathbb{E}\Big[\Big|\int_0^T X(s)\cdot dM(s)\Big|^2\Big]=\mathbb{E}\Big[\int_0^T |X(s)|^2\,d\la M\ra(s) \Big],
\end{equation}
where $\la M\ra(s)$ is the quadratic variation of $M$; if $M=W$ is the Brownian motion, then $\la W\ra(s)=s$. 

For generic processes $X\in M^2([0,T];M)$, that is, those satisfying $\mathbb{E}[\int_0^T|X(s)|^2\,d\la M\ra(s)]<\infty$, one can show that they can be approximated by elementary processes in the sense that there exists a sequence $X_n$ elementary such that
\begin{equation*}\lim_{n\to\infty}\mathbb{E}\Big[\int_0^T|X(s) - X_n(s)|^2\,d\la M\ra(s) \Big]=0:
\end{equation*}
then, one can define the $\hat{\text{I}}$to stochastic integral of $X$ with respect to $M$ as the limit of stochastic integrals for $X_n$, and the property \eqref{ito.isometry} still holds for $X$. One can then show that still, the $\hat{\text{I}}$to stochastic integral of $M^2([0,T];M)$ processes is a continuous $(\mF_t)_t-$adapted martingale with zero mean.
We recall here the Burkholder-Davis-Gundy Inequality which will be used massively for our estimates:
\begin{prop}[Burkholder-Davis-Gundy Inequality]\label{prop:Burkholder.Davis.Gundy.Inequality} Let $p>0$, let $X\in M^2([0,T];M)$ a square-integrable process with respect to a martingale $M$ with finite quadratic variation. Then, there exist two constants $c(p),C(p)>0$, depending only on $p$, such that
\begin{equation}\label{Doob.maximal.inequality} c(p)\cdot\mathbb{E}\Big[\int_0^T |X(s)|^2\,d\la M\ra(s) \Big]^{\frac{p}{2}}\le\mathbb{E}\Big[\sup_{t\in[0,T]}\Big|\int_0^t X(s)\cdot dM(s) \Big|^p \Big]\le C(p)\cdot\mathbb{E}\Big[\int_0^T |X(s)|^2\,d\la M\ra(s) \Big]^{\frac{p}{2}}.
\end{equation}
\end{prop}
If $M=W$ is the Brownian motion, then we will refer to square-integrable processes space simply as $M^2([0,T])$.

\medskip

Given any continuous process $X$ and $F_1,F_2\in M^2([0,T])$, we say that $X$ has a stochastic differential $dX(t)=F_1(t)dt + F_2(t)\cdot dW(t)$ if
\begin{equation*}X(t) = X(0) +  \int_0^T F_1(t)\,dt + \int_0^T F_2(t)\cdot dW(t).
\end{equation*}
We have the $\hat{\text{I}}$to Formula:
\begin{prop}[$\hat{\text{I}}$to Formula]\label{prop:Ito.formula}
Let $W=(W_1,\dots,W_d)$ be a $d-$dimensional Brownian motion, let $F\in C^{1,2}(\R_+\times\R^d;\R)$. Then, if $dX(t)=F_1(t)dt +  F_2(t)\cdot dW(t)$, one has
\begin{align}d(F(t,X(t)))=&\pa_tF(t,X(t))dt + \grad F(t,X(t))\cdot dX(t) +\frac12\text{tr}[H(F)F_2 F_2^*](t,X(t))dt, \label{Ito.formula}
\end{align}
where $\text{tr}$ is the trace operator and $H(F)$ is the Hessian matrix of $F$.
\end{prop}

\subsection{Wasserstein spaces}

In what follows, we will introduce Wasserstein spaces for measures on complete metric spaces and we will recall some properties and estimates that will be used in the following sections.

Now, for any $p\ge1$, let us denote by $\mW_p(B)$ the $p-$Wasserstein space over a complete metric space $(B,d_B)$; that is, the space of Borel probability measures on $B$ with finite $p-$th momentum, namely
\begin{equation}\label{p.Wasserstein.space}\mW_p(B):=\{\mu\,|\,\mu\,\text{is a probability measure},\,M_p^p(\mu;x_0):=\int_B d_B^p(x,x_0)\,d\mu(x)<\infty\};
\end{equation}
we recall that by Triangle Inequality, $\mW_p(B)$ does not actually depend on the choice of $x_0\in B$; then, if $B$ is a Banach space, then we set $x_0:=0$. We recall also that the Wasserstein spaces are metrizable, precisely we have the following fact.

For any $\mu, \nu\in\mW_p(B)$, a probability Borel measure $\g$ on $B\times B$ is said to be a coupling of $\mu,\nu$ if for any Borel set $A\subset B$, $\g(A\times B)=\mu(A)$ and $\g(B\times A)=\nu(A)$. Denoting by $\Pi(\mu,\nu)$ the set of all couplings of $\mu,\nu$, we define the $p-$Wasserstein metric as the one such that
\begin{equation}\label{p.Wasserstein.metric}\mW_p^p(\mu,\nu):=\inf_{\g\in\Pi(\mu,\nu)}\int_{B\times B}d_B^p(x,y)\,d\g
\end{equation}
It is known that the infimum is reached by the so called optimal coupling; see for instance Theorem 4.1, \cite{Villani}. Equivalently, the $p-$Wasserstein metric turns out to be
\begin{equation}\label{p.Wasserstein.metric.1}\mW_p^p(\mu,\nu)=\inf\{\mathbb{E}[d_B^p(X,Y)]\colon\,\text{Law}(X)=\mu,\,\text{Law}(Y)=\nu \}\end{equation}

A special case is covered when $p=1$, for which the metric $\mW_1(\mu,\nu)$ can be characterized through Kantorovich-Rubinstein duality:
\begin{equation}\label{Kantorovich.Rubinstein.duality}\mW_1(\mu,\nu)=\sup\Big\{\Big|\int_B\ph\,d\mu - \int_B\ph\,d\nu\Big|\,\colon\,\ph\in\text{Lip(B;$\R$)},[\ph]_{\text{Lip}}\le1\Big\},
\end{equation}
where $\displaystyle[\ph]_{\text{Lip}}:=\sup_{x\ne y}\frac{|\ph(y)-\ph(x)|}{d_B(y,x)}$.
Moreover, one has that for all $p\le q$ and for all $\mu,\nu\in\mW_q(B)$, \begin{equation}\label{Wasserstein.sort}\mW_p(\mu,\nu)\le\mW_q(\mu,\nu)\end{equation}
which implies that $\mW_1$ is the weakest Wasserstein distance.

Furthermore, if we have two collections of $N$ random variables, namely $(X_n)_{n=1}^N,(Y_n)_{n=1}^N\subset B$, denoting by $\mu_N(Z):=\frac{1}{N}\sum_{n=1}^N\delta_{Z_n}$ the related empirical measure of a generic collection $Z=(Z_n)_{n=1}^N$ of random variables, we get that
\begin{align}&M_p(\mu_N(X))\le\Big(\frac{1}{N}\sum_{n=1}^N d_B^p(X_n,0) \Big)^{\frac{1}{p}}\le\max_{1\le n\le N}d_B(X_n,0) \label{Mp.norm.empirical.measure}
\\&\mW_p(\mu_N(X),\mu_N(Y))\le\Big(\frac{1}{N}\sum_{n=1}^N d_B^p(X_n,Y_n)\Big)^{\frac{1}{p}}\le\max_{1\le n\le N}d_B(X_n,Y_n);\label{Wasserstein.empirical.measure}
\end{align}
recalling the introduction, the empirical measures correspond to the herd statistical distributions.

If $d\ge1$ and $(B,d_B):=(C([0,T];\R^d),\|\cdot\|_\infty)$, for all $\mu\in\mW_p(B)$ we define the distance
\begin{equation}\label{d.continuous.measure.distance}d(\mu,\nu):=\sup_{t\in[0,T]}\mW_p(\mu(t),\nu(t)).
\end{equation}

We now observe some useful estimates we will use in the next sections.

\begin{remark}Let $(B,\|\cdot\|_B)$ be a Banach space and $f\colon\,B\longmapsto\R$ be a function such that there exists $L>0$ for which for all $x_1,x_2\in B$
\begin{equation*}|f(x_2)-f(x_1)|\le L\|x_2 - x_1\|_B.
\end{equation*}
Then, if we set for all $x\in B$
\begin{equation*}f*\mu_t(x):=\int_B f(x-y)\,d\mu_t(y),
\end{equation*}
then for all $x_1,x_2\in B$, $f*\mu_t$ satisfies
\begin{equation}|f*\mu_t(x_2) - f*\mu_t(x_1)|\le L\|x_2-x_1\|_B. \label{convoluted.function.lipschitz}
\end{equation}
\end{remark}

\begin{remark}In notations and assumptions of point $(i)$, for any $\mu\in C([0,T];\mW_1(\R^d))$ let us consider the linear operator $F\colon\,C([0,T];\mW_1(\R^d))\longmapsto C([0,T];\R^d)$
\begin{equation*}(F(\mu))(t):=f*\mu_t.\end{equation*}
Then, for all $\mu_1,\mu_2\in C([0,T];\mW_1(\R^d))$, $F$ satisfies
\begin{align}&|F(\mu_2)(t) - F(\mu_1)(t)|\le L\cdot\mW_1(\mu_2(t),\mu_1(t)), \label{convoluted.operator.lipschitz.pointwise}
\\&\|F(\mu_2)-F(\mu_1)\|_\g\le L\cdot d_{1,\g}(\mu_2,\mu_1)
\label{convoluted.operator.lipschitz}\end{align}
where we considered
\begin{align}
&\|f_2-f_1\|_\g:=\sup_{t\in[0,T]}e^{-\g t}|f_2(t)-f_1(t)|,\qquad\forall f_1,f_2\in C([0,T];\R^d), \label{decay.norm.continuous.function}
\\&d_{p,\g}(\mu_2,\mu_1):=\sup_{t\in[0,T]} e^{-\g t}\mW_p(\mu_2(t),\mu_1(t)),\qquad\forall\mu_1,\mu_2\in C([0,T];\mW_p(\R^d)),\,\forall p\ge1 \label{decay.distance.measures}
\end{align}
We remark here that \eqref{decay.norm.continuous.function} is a complete norm for $C([0,T];\R^d)$, and similarly \eqref{decay.distance.measures} is a complete metric for $C([0,T];\mW_p(\R^d))$.

\end{remark}

To achieve the mean-field limit (Theorem \ref{thm:propagation.of.chaos}), we provide here a characterization of convergence in Wasserstein spaces, and the classical Law of Large Numbers and the quantitative one proved in \cite{Fournier.Guillin}:

\begin{prop}\label{prop:technical.estimates}

Let $(B,d_B)$ a complete metric space. Then, the following facts hold.

\medskip

$(i)$(Theorem 6.9, \cite{Carmona.Delarue}) Let $p\ge1$, let $(\mu_n)_{n\in\N},\,\mu\in\mW_p(B)$ be probability measures on $B$. Then, one has that $\lim_{n\to\infty}\mW_p(\mu,\mu_n)=0$ if and only if for all $\ph\in C_b(B;\R)$, $\la\mu_n,\ph\ra\to\la\mu,\ph\ra$ and $\la\mu_n,d_B^p(x,x_0)\ra\to\la\mu,d_B^p(x,x_0)\ra$ as $n\to\infty$. 

\medskip

$(ii)$(Strong Law of Large Numbers) Let $(X_n)$ be a sequence of independent, identically distributed random variables on a probability space $(\Om,\mF,\mathbb{P})$ and taking values on $B$, such that for all $n\in\N$, $\mathbb{E}[X_n]=m$. Let us call $X_N:=\frac{1}{N}\sum_{n=1}^N X_n$. Then, $\mathbb{P}-$almost surely
\begin{equation}\label{strong.law.of.large.numbers}\lim_{N\to\infty}X_N=m.
\end{equation}

\medskip

$(iii)$(Theorem 1, \cite{Fournier.Guillin}) Let $q>0$, let $\mu\in\mW_q(\R^d)$ such that there exists some $p>q$ for which $M_p(\mu)<+\infty$. Let $(X_n)_{n\in\N}$ be a sequence of independent, identically distributed random variables with distribution $\mu$. For any $N\in\N$, let $\mu_N:=\frac{1}{N}\sum_{n=1}^N\delta_{X_n}$ be the associated empirical measure. Then,
\begin{equation}\label{law.of.large.numbers}\mathbb{E}[(\mW_q^q(\mu_N,\mu))]\lesssim_{q,d,p} M_p^{\frac{q}{p}}(\mu)\cdot\begin{cases}N^{-\frac12} + N^{-\frac{p-q}{p}}\qquad\qquad\qquad\quad\text{if}\quad q>\frac{d}{2},\,p\ne2q,
\\N^{-\frac12}\log(1+N) + N^{-\frac{p-q}{p}}\qquad\,\text{if}\quad q=\frac{d}{2},\,p\ne2q,
\\N^{-\frac{q}{d}} + N^{-\frac{p-q}{p}}\qquad\qquad\qquad\quad\text{if}\quad q\in(0,\frac{d}{2}),\,p\ne\frac{d}{d-q}.
\end{cases}
\end{equation}
    
\end{prop}

\begin{remark}\label{remark:non.quantitative.law.large.numbers}   In the notations of Proposition \ref{prop:technical.estimates}, if $\mu$ is the law of a random variable $X\in L^p(\Om;B)$ and $(X_n)$ is a sequence of its independent copies, the point $(ii)$ implies together with the point $(i)$ that for all $p\ge1$,
\begin{equation}\label{strong.law.large.numbers.wasserstein.1}\lim_{N\to\infty}\mW_p(\mu,\mu_N)=0.
\end{equation}
Now, let us denote by $\delta_{x_0}$ the Dirac Delta concentrated at $x_0\in B$. We observe that 
\begin{equation*}\mW_p^p(\mu,\mu_N)\lesssim_p\mW_p^p(\mu,\delta_{x_0}) + \mW_p^p(\delta_{x_0},\mu_N)\lesssim_p\underbrace{\mathbb{E}[d_B^p(X,x_0)]}_{=:(I)} + \underbrace{\sup_{N\in\N}\frac{1}{N^p}\sum_{n=1}^N d_B^p(X_n,x_0)}_{=:(II)};
\end{equation*}
since $(I)$ is bounded by assumption and $(II)$ is $L^1-$bounded by point $(i)$ of Proposition \ref{prop:technical.estimates}, we get by Lebesgue Dominated Convergence Theorem that
\begin{equation}\label{strong.law.large.numbers.wasserstein.2}\lim_{N\to\infty}\mathbb{E}[\mW_p^p(\mu,\mu_N)]=0.
\end{equation}

As for \eqref{law.of.large.numbers}, there is no quantitative estimate of convergence when $q=p$: indeed, in \cite{Fournier.Guillin} there is a counterexample for this, see Example $(c)$, page 2.    
\end{remark}

\subsection{Assumptions on interactions, noises and controls}

In the next sections, we will always consider the following assumptions.

\begin{ass}\label{assumptions} Let $(\Om,\mF,(\mF_t)_{t\in[0,T]},\mathbb{P})$ be a filtration over a complete probability space $(\Om,\mF,\mathbb{P})$.
Given $p\ge2$, $d\in\{1,2,3\}$, $N,M\in\N$, $L>0$ and $T>0$, let us denote by $x:=(x_1,\dots,x_d)\in\R^d$ a generic vector in $\R^d$ and, with a small abuse of notations, by $\displaystyle|x|:=\max_{1\le i\le d}|x_i|$ (and we extend this definition to whatever $\R^k$, for any $k\in\N$).

\medskip

$(i)$(Interaction kernels) Given $i\in\{1,2\}$, the kernel interactions 
\begin{equation*}H_i,K_i\colon\,\R^d\longmapsto\R^d\end{equation*}
are such that for any $x,y\in\R^d$, 
\begin{equation}\label{K.L.Lipschitz}\max_{i=1,2}\,\{|H_i(y)-H_i(x)|,\,|K_i(y)-K_i(x)|\}\le L|y-x|.
\end{equation}

\medskip

$(ii)$(Idiosyncratic noise function) The idiosyncratic noise function 
\begin{equation*}\s_*^{(i)}\colon\,[0,T]\times(\R^d)^M\times\R^d\times\mW_1(\R^d)\longmapsto\R^{d\times d}\end{equation*}
in equations \eqref{SDE.finite}, \eqref{SDE.mean.field} satisfies
\begin{align}&\s_*^{(i)}=\s_*^{(i)}(t,y,x,\nu),\label{idiosyncratic.noise.shape.mean.field}
\\&\s_*^{(i)}\,\,\text{is continuous}, \label{idiosyncratic.noise.is.continuous}
\\&|\s_*^{(i)}(t,y_2,x_2,\nu_2) - \s_*^{(i)}(t,y_1,x_1,\nu_1)|\le L\cdot[|y_2-y_1| + |x_2-x_1| + \mW_1(\nu_2,\nu_1)], \label{idiosyncratic.noise.is.lipschitz}
\end{align}

\medskip

$(iii)$(Common noise function) The common noise function \begin{equation*}\s_*^{(c)}\colon\,[0,T]\times(\R^d)^M\times\R^d\times\mW_1(\R^d)\longmapsto\R^{d\times d}\end{equation*}
in equations \eqref{SDE.finite}, \eqref{SDE.mean.field} satisfies
\begin{align}&\s_*^{(c)}=\s_*^{(c)}(t,y,x,\nu),\label{common.noise.shape.mean.field}
\\&\s_*^{(i)}\,\,\text{is continuous}, \label{common.noise.is.continuous}
\\&|\s_*^{(i)}(t,y_2,x_2,\nu_2) - \s_*^{(i)}(t,y_1,x_1,\nu_1)|\le L\cdot[|y_2-y_1| + |x_2-x_1| + \mW_1(\nu_2,\nu_1)], \label{common.noise.is.lipschitz}
\end{align}

\medskip

$(iv)$(Control function) The control function 
\begin{equation*}u\colon\,[0,T]\times(\R^d)^M\times\mW_1(\R^d)\times\longmapsto L^p(\Om;\R^M)\end{equation*}
is a $\mF_t-$adapted, progressively measurable process, and $\mathbb{P}-$almost surely it satisfies
\begin{align}&u=u(t,Y,\nu),\label{control.function.shape}
\\&f(t):=u(t,Y,\nu)\in L^1([0,T];L^p(\Om)), \label{u.integrable.in.time}
\\&|u(t,Y_2,\nu_2)-u(t,Y_1,\nu_1)|\le L\cdot[|Y_2-Y_1| + \mW_1(\nu_2,\nu_1)]. \label{u.lipschitz.in.measure}
\end{align}
   
\end{ass}

\section{Well-posedness of discrete and mean-field equations}

We start by defining the strong solutions for \eqref{SDE.finite}-\eqref{Initial.data.finite} and \eqref{SDE.mean.field}-\eqref{Initial.data.mean.field}.

\begin{definition}\label{Def.conditional.solutions}

Let us fix a filtration $(\Om,\mF,(\mF_t)_{t\in[0,T]},\mathbb{P})$. Let $p\ge2$.

\medskip

$(i)$(Discrete problem)
Given the system \eqref{SDE.finite}-\eqref{ODE.finite} with initial data 
\begin{equation*}(X^{(0,N)},Y^{(0,N)})\in(L^p(\Om;\R^d))^N\times(\R^d)^M,\end{equation*}
we say that 
\begin{equation*}(X^{(N)}(t),Y^{(N)}(t))\in C([0,T];(L^p(\Om;\R^d))^{N})\times C([0,T];(L^p(\Om;\R^d))^{M})\end{equation*}
is a strong solution if for any $(\mF_t)_t-$adapted $d\times(N+1)-$dimensional standard Brownian motion $\mathbf{W}:=(W^{(c)},W_{1}^{(i)},\dots,W_N^{(i)})$ with independent components, for all $1\le n\le N$, for all $1\le m\le M$ and for all $t\in[0,T]$, $\mathbb{P}-$almost surely
\begin{align}&X_n^{(N)}(t)=X_n^{(0,N)} + \int_0^t\Big\{H_1*\mu_N(s)(X_n^{(N)}(s)) + \frac{1}{M}\sum_{m=1}^M K_1(Y_m^{(N)}(s)-X_n^{(N)}(s))\Big\}\,ds \notag
\\&\qquad\qquad+ \int_0^t\s_*^{(i)}(s,Y^{(N)}(s),X_n^{(N)}(s), \mu_N(s))dW_n^{(i)}(s) \notag
\\&\qquad\qquad+ \int_0^t\s_*^{(c)}(s,Y^{(N)}(s),X_n^{(N)}(s),\mu_N(s))dW^{(c)}(s), \label{strong.solution.X.discrete.common}
\\&Y_m^{(N)}(t)=Y_m^{(0,N)} + \int_0^t\Big\{K_2*\mu_N(s)(Y_m^{(N)}(s)) + \frac{1}{M}\sum_{j=1}^M H_2(Y_m^{(N)}(s) - Y_j^{(N)}(s))\,ds \notag
\\&\qquad\qquad+ \int_0^t u_m(s,Y^{(N)}(s),\mu_N(s))\Big\}\,ds, \label{strong.solutions.Y.discrete.common}
\end{align}
and both $X,Y$ are $(\mF_t)_t-$adapted processes. Here, $W^{(c)}$ is the common noise Brownian motion, while $W_1^{(i)},\dots W_N^{(i)}$ are the idiosyncratic noise Brownian motions.

\medskip

$(ii)$(Mean-field problem) Given the system \eqref{SDE.mean.field}-\eqref{ODE.mean.field} with only one SDE, with initial data  
\begin{equation*}(X^{(0)},Y^{(0)})\in L^p(\Om;\R^d)\times(\R^d)^M,\end{equation*}
and initial conditional law $\mu_0^{(i)}$,
we say that $(X(t),Y(t))\in C([0,T];L^p(\Om;\R^d))\times C([0,T];(\R^d)^M)$ is a strong solution if for any $(\mF_t)_t-$adapted $2d-$dimensional standard Brownian motion $\mathbf{W}:=(W^{(c)},W^{(i)})$ with independent components, for all $1\le m\le M$ and for all $t\in[0,T]$, $\mathbb{P}-$almost surely
\begin{align}&X(t)=X^{(0)} + \int_0^t\Big\{H_1*\mu^{(i)}(s)(X(s)) + \frac{1}{M}\sum_{m=1}^M K_1(Y_m(s)-X(s))\Big\}\,ds \notag
\\&\qquad\quad+\int_0^t\s_*^{(i)}(s,Y(s),X(s),\mu^{(i)}(s))dW^{(i)}(s)+ \int_0^t\s_*^{(c)}(s,Y(s),X(s),\mu^{(i)}(s))dW^{(c)}(s)
, \label{strong.solution.X.mean.field.common}
\\&Y_m(t)=Y_m^{(0)} + \int_0^t\Big\{K_2*\mu^{(i)}(s)(Y_m(s)) + \frac{1}{M}\sum_{j=1}^M H_2(Y_m(s) - Y_j(s))\,ds \notag
\\&\qquad\quad+ \int_0^t u_m(s,Y(s),\mu^{(i)}(s))\Big\}\,ds. \label{strong.solutions.Y.mean.field.common}
\end{align}
Moreover, $X$ is a $(\mF_t)_t-$adapted process, $Y$ is a $(\mF_t)_t-$adapted process and $\mu^{(i)}(t)$, which is the conditional law of $X(t)$ given $W^{(c)}(t)$, is a $(\mF_t)_t-$adapted continuous process such that $\mathbb{P}-$almost surely and for all $t\in[0,T]$, $\mu^{(i)}(t)\in\mW_p(\R^d)$.
Here, $W^{(c)}$ is the common noise Brownian motion, while $W^{(i)}$ is the mean-field idiosyncratic noise Brownian motion.
\end{definition}

\begin{remark}\label{rem:conditional.law} For any fixed random measure flow $(\nu_t)_{t\in[0,T]}$, if we assume that the stochastic differential equation
\begin{align*}&X(t)=X^{(0)} + \int_0^t\Big\{H_1*\nu(s)(X(s)) + \frac{1}{M}\sum_{m=1}^M K_1(Y_m(s)-X(s))\Big\}\,ds \notag
\\&\qquad\quad+\int_0^t\s_*^{(i)}(s,Y(s),X(s),\nu(s))dW^{(i)}(s)+ \int_0^t\s_*^{(c)}(s,Y(s),X(s),\nu(s))dW^{(c)}(s)
,
\\&Y_m(t)=Y_m^{(0)} + \int_0^t\Big\{K_2*\nu(s)(Y_m(s)) + \frac{1}{M}\sum_{j=1}^M H_2(Y_m(s) - Y_j(s))\,ds \notag
\\&\qquad\quad+ \int_0^t u_m(s,Y(s),\nu(s))\Big\}\,ds. 
\end{align*} 
has a unique strong solution, then there exists a measurable map $\Phi=(\Phi_1,\Phi_2)\colon\,C([0,T];\R^{2d})\longmapsto C([0,T];\R^{2d})$ such that the pair $(X,Y)=(\Phi_X,\Phi_Y)(W^{(c)},W^{(i)})$. Therefore, we can express the law of $X$ as $\text{Law}(X(t))(A)=\mathbb{E}[\mathbf{1}_A(\Phi_X(W^{(c)},W^{(i)})(t))]=\mathbb{E}[\mathbb{E}[\mathbf{1}_A(\Phi_X(W^{(c)},W^{(i)})(t))|\mF_t^{(c)}]]$, where $(\mF_t^{(c)})_{t\in[0,T]}$ is the natural filtration of $W^{(c)}$. By Doob-Dynkin Lemma (see \cite{BHL}, Lemma 2.11), we have that $\mathbb{E}[\mathbf{1}_A(\Phi_X(W^{(c)},W^{(i)})(t))|\mF_t^{(c)}]=F_t(A,W^{(c)}(t))$ for some function $\mB(\R^d)-$measurable function 
\begin{equation*}
F_t\colon\,(A,f) \in\mF_t\times C([0,T];\R^d) \longmapsto \mathbb{E}[\mathbf{1}_A(\Phi_X(f,W^{(i)})(t))] \in \R,
\end{equation*}
Since $W^{(c)}(t)$ is a $\mF_t-$measurable random variable, then we define the $\mF_t-$measurable random variable
\begin{equation*}
\text{Law}(X(t)|\mF_t^{(c)})\colon\,(B,\omega) \in \mB(\R^d)\times\Om \longmapsto F_t(X(t)^{-1}(B),W^{(c)}(t,\omega)) \in \R,
\end{equation*}
Since $\R^d$ is a Polish space, Theorem 2.12 in \cite{BHL} guarantees that up to null-measure sets, $\text{Law}(X(t)|\mF_t^{(c)})$ is the unique regular conditional law of $X(t)$ given $\mF_t^{(c)}$, that is, it is the unique function $Q_t\colon\,(B,\om)\in\mB(\R^d)\times\Om\longmapsto[0,1]$ such that

\medskip

$(i)$ $Q_t(B,\cdot)$ is $\mB(\R^d)-$measurable for all $B\in\mB(\R^d)$,

$(ii)$ $Q_t(\cdot,\om)$ is a probability measure on $(\R^d,\mB(\R^d))$ for all $\om\in\Om$,

$(iii)$ $Q_t(B,\om)=F_t(X(t)^{-1}(B),W^{(c)}(t,\om))$.

\medskip

This identity can be then taken as the definition of the conditional law of $X(t)$ given the common noise $W^{(c)}(t)$, which turns out to be a $\mF_t^{(c)}-$measurable random variable. One also has that $\text{Law}(X|\mF^{(c)})$ is a $(\mF_t^{(c)})_t-$adapted process holding a continuous version, and then it defines a unique random flow of measures: see for instance \cite{Cherny}. Moreover, if we take different independent, identically distributed initial data, their corresponding solutions preserve these two properties.
    
\end{remark}

\begin{remark}\label{remark:no.common.noise.gives.deterministic.law} We notice that if there is no common noise, that is, $\s_*^{(c)}\equiv0$, then the solutions are simply $\mF_t^{(i)}-$adapted, and so there is no dependence by $\Om^{(c)}$ for the mean-field law $\mu^{(i)}$; in other words, we recover the setting provided in \cite{ACS}, apart from the multiplicative idiosyncratic noises and the dependence of controls function on the herders.
\end{remark}

In the next sections, for simplicity of notations, when performing estimates we will often neglect the dependence on time and $\om\in\Om$.

We recall also the Gr\"onwall Inequality which will be massively useful in doing estimates:
\begin{prop}[Gr\"onwall Inequality]\label{prop:Gronwall.inequality}
Let $w,v\colon\,[a,b]\longmapsto\R$ be measurable functions, with $v$ bounded, and let $c\ge0$ a constant such that $v(t)\le c + \int_a^t w(s)v(s)\,ds$ for all $t\in[a,b]$. Then, for all $t\in[a,b]$
\begin{equation}\label{Gronwall.inequality}v(t)\le c\exp\Big(\int_a^t w(s)\,ds\Big).
\end{equation}

\end{prop}
In this section, we will show the well-posedness of \eqref{SDE.finite}-\eqref{Initial.data.finite} and \eqref{SDE.mean.field}-\eqref{Initial.data.mean.field}.
To start with, we will consider the discrete problem \eqref{SDE.finite}-\eqref{Initial.data.finite}.

\begin{lemma}\label{lemma:finite.well.posedness}
Given Assumptions \ref{assumptions}, for any $p\ge2$ one has that system \eqref{SDE.finite}-\eqref{Initial.data.finite} has a unique strong solution in the sense of Definition \ref{Def.conditional.solutions}, point $(i)$. Moreover, if $(X^{(N)},Y^{(N)})$ is the unique strong solution of \eqref{SDE.finite}-\eqref{Initial.data.finite} and $\mu_N$ the associated empirical measure, one has that:
\begin{align}&\sup_{t\in[0,T]}\max\{\mathbb{E}[|(X^{(N)}(t),Y^{(N)}(t))|^p],\mathbb{E}[M_p^p(\mu_N(t))]\}\le C,\label{discrete.solution.boundedness}\end{align}
where $C=C(p,L,T,H_2(0),K_2(0),X^{(0,N)},Y^{(0,N)},u)$.

\end{lemma}

\begin{proof}

Let us consider the map $\Phi$ given by
\begin{align}&\Phi(F,G):=(\Phi_1(F,G),\Phi_2(F,G)), \label{Phi.map.finite}
\\&(\Phi_1(F,G))_n(t):=X_n^{(0,N)} + \int_0^t\Big\{H_1*\mu_N(s)(F_n(s)) + \frac{1}{M}\sum_{j=1}^M K_1(F_n(s)-G_j(s))\Big\}\,ds \notag\\&\qquad\qquad\quad\qquad+\int_0^t\s_*^{(i)}(s,G(s),F_n(s),\mu_N(s))dW_n^{(i)}(s) + \int_0^t\s_*^{(c)}(s,G(s),F_n(s),\mu_N(s))dW^{(c)}(s)
\label{Phi.1},
\\&(\Phi_2(F,G))_m(t):=Y_m^{(0,N)} + \int_0^t\Big\{K_2*\mu_N(s)(G_m(s)) + \frac{1}{M}\sum_{j=1}^M H_2(G_m(s) - G_j(s))\Big\}\,ds \notag
\\&\qquad\qquad\qquad\qquad+ \int_0^t u_m(s,G(s),\mu_N(s))\,ds;\label{Phi.2}
\end{align}
here, we are denoting by 
\begin{equation*}(\mu_N)_s:=(\mu_N)_s[F]:=\frac{1}{N}\sum_{j=1}^N\delta_{F(s)}.
\end{equation*}
We equip $B:=C([0,T];L^p(\Om;(\R^d)^{N\times M}))$ with the norm
\begin{equation*}\|\eta\|_{p,\g}:=\Big(\sup_{t\in[0,T]}e^{-p\g t}\mathbb{E}[|\eta(t)|^p]\Big)^{\frac{1}{p}}.
\end{equation*}

\medskip

We divide the proof into two steps.

\medskip

$(i)$ \emph{The map $\Phi$ maps $B$ into itself.}

\medskip

Let us fix $(F,G)\in B$. Let us start by estimating $\Phi_1(F,G)$: for all $n$,\,$1\le n\le N$, we have
\begin{align}|(\Phi_1(F,G))_n|
&\le|X_n^{(0,N)}| \notag
\\&+ \underbrace{\int_0^t|H_1*\mu_N(F_n)|\,ds}_{=:(I)} + \underbrace{\frac{1}{M}\sum_{j=1}^M\int_0^t|K_1(F_n - G_j)|\,ds}_{=:(II)} \notag
\\&+\Big|\int_0^t\s_*^{(i)}(s,G,F_n,\mu_N)dW_n^{(i)}(s) \Big| +  \Big|\int_0^t\s_*^{(c)}(s,G,F_n,\mu_N)dW^{(c)}(s) \Big|
\label{first.estimate}
\end{align}
Let us have a look at $(I)$. By \eqref{K.L.Lipschitz}, \eqref{convoluted.function.lipschitz} and \eqref{Mp.norm.empirical.measure} one has
\begin{align*}|H_1*\mu_N(F_n)(s)|
&\le|H_1*\mu_N(F_n) - H_1*\mu_N(0)|(s) + |H_1*\mu_N(0)|(s)
\\&\le L\cdot|F(s)| + \int_{\R^d}|H_1(-y)|\,d\mu_N(y)
\\&\le L\cdot|F(s)| + L\int_{\R^d}|y|\,d\mu_N(y) + |H_1(0)|
\\&\le 2L\cdot|F(s)| + |H_1(0)|,
\end{align*}
so that
\begin{equation*}(I)\lesssim_{L,H_1(0),T}\int_0^t|F(s)|\,ds.
\end{equation*}
As for $(II)$, by \eqref{K.L.Lipschitz} one has
\begin{align*}|K_1(F_n - G_j)|
&\le|K_1(F_n - G_j) - K_1(0)| + |K_1(0)|
\\&\le L|F_n - G_j| + |K_1(0)|
\\&\le L[|F| + |G|] + |K_1(0)|, 
\end{align*}
so that
\begin{equation*}(II)\lesssim_{L,K_1(0),T}\int_0^t[|F(s)| + |G(s)|]\,ds.
\end{equation*}
Plugging these estimates into \eqref{first.estimate}, we get
\begin{align*}|(\Phi_1(F,G))_n|
&\lesssim_{C}|X^{(0,N)}| + \int_0^t[|F(s)| + |G(s)|]\,ds 
\\&+\Big|\int_0^t\s_*^{(i)}(s,G,F_n,\mu_N)dW_n^{(i)}(s) \Big| +  \Big|\int_0^t\s_*^{(c)}(s,G,F_n,\mu_N)dW^{(c)}(s) \Big|,
\end{align*}
where $C=C(L,T,H_1(0),K_1(0))$ is a positive constant. Exponentiating both sides by $p$ and taking the expectation, we get by \eqref{Doob.maximal.inequality}
\begin{align*}\mathbb{E}[|\Phi_1(F,G)|^p]
&\lesssim_{C,p}\mathbb{E}[|X^{(0,N)}|^p] + \int_0^t[\mathbb{E}[|F(s)|^p] + \mathbb{E}[|G(s)|^p]]\,ds
\\&+\mathbb{E}\Big[\sup_{s\in[0,t]}\Big|\int_0^s\s_*^{(i)}(z,G,F_n,\mu_N)dW_n^{(i)}(z) \Big|^p\Big] 
\\&+ \mathbb{E}\Big[\sup_{s\in[0,t]}\Big|\int_0^s\s_*^{(c)}(z,G,F_n,\mu_N)dW^{(c)}(z) \Big|^p\Big]
\\&\lesssim_{C,p}\mathbb{E}[|X^{(0,N)}|^p] + \int_0^t[\mathbb{E}[|F(s)|^p] + \mathbb{E}[|G(s)|^p]]\,ds 
\\&+\int_0^t\mathbb{E}[|\s_*^{(i)}(s,G(s),F_n(s),\mu_N(s))|^p]\,ds + \int_0^t\mathbb{E}[|\s_*^{(c)}(s,G(s),F_n(s),\mu_N(s))|^p]\,ds.
\end{align*}
Similarly, using Assumptions \ref{assumptions}, for $\Phi_2(F,G)$ it holds that
\begin{align*}\mathbb{E}[|\Phi_2(F,G)|^p]
&\lesssim_{C,p}|Y^{(0,N)}|^p + \int_0^t[\mathbb{E}[|F(s)|^p] + \mathbb{E}[|G(s)|^p]]\,ds + \underbrace{\mathbb{E}\Big[\Big|\int_0^t u(s,G(s),\mu_N(s))\,ds\Big|^p\Big]}_{=:(I)}
\end{align*}
where this time $C=C(L,T,H_2(0),K_2(0))$.
We want to estimate $(I)$. By \eqref{u.lipschitz.in.measure}, $u\in L^p(\Om^{(c)})$ and  \eqref{u.integrable.in.time}, we have
\begin{align*}(I)
&\lesssim_p\int_0^t\mathbb{E}\Big[|u(s,G(s),\mu_N(s))-u(s,0,\delta_0)|^p\Big]\,ds + \mathbb{E}\Big[\Big|\int_0^t u(s,0,\delta_0)\,ds\Big|^p\Big]
\\&\lesssim_{p,L,u}\int_0^t\mathbb{E}\Big[|G(s)|^p + \mW_1^p(\mu_N(s),\delta_0) \Big]\,ds + 1
\end{align*}

Then, one finds out that $\Phi(F,G)\in B$ by the fact that $X^{(0,N)}\in (L^p(\Om;\R^d))^N$ and by \eqref{idiosyncratic.noise.is.continuous}, \eqref{u.integrable.in.time} and \eqref{u.lipschitz.in.measure}.

\medskip

$(ii)$ \emph{The map $\Phi$ is a contraction.}

\medskip

Let us take $(F_1,G_1),(F_2,G_2)\in B$, let us call $\mu_{1,N},\mu_{2,N}$ the empirical measures associated to $F_1,F_2$ respectively. Let us take a look at $\Phi_1$: we have
\begin{align}|(\Phi_1(F_2,G_2))_n-(\Phi_1(F_1,G_1))_n|
&\le\int_0^t|H_1*\mu_{2,N}(F_{2,n}) - H_1*\mu_{1,N}(F_{1,n})|\,ds \notag
\\&+\frac{1}{M}\sum_{j=1}^M\int_0^t|K_1(F_{2,n} - G_{2,j}) - K_1(F_{1,n} - G_{1,j})|\,ds \notag
\\&+\Big|\int_0^t[\s_*^{(i)}(s,G_2,F_{2,n},\mu_{2,N}) - \s_*^{(i)}(s,G_1,F_{1,n},\mu_{1,N})]\,dW_n^{(i)}(s)\Big| \notag
\\&+\Big|\int_0^t[\s_*^{(c)}(s,G_2,F_{2,n},\mu_{2,N}) - \s_*^{(c)}(s,G_1,F_{1,n},\mu_{1,N})]\,dW^{(c)}(s)\Big| \notag
\\&\le\underbrace{\int_0^t|H_1*\mu_{2,N}(F_{2,n}) - H_1*\mu_{1,N}(F_{2,n})|\,ds}_{=:(I)} \notag
\\&+\underbrace{\int_0^t|H_1*\mu_{1,N}(F_{2,n}) - H_1*\mu_{1,N}(F_{1,n})|\,ds}_{=:(II)} \notag
\\&+\underbrace{\frac{1}{M}\sum_{j=1}^M\int_0^t|K_1(F_{2,n} - G_{2,j}) - K_1(F_{1,n} - G_{1,j})|\,ds}_{=:(III)} \notag
\\&+\Big|\int_0^t[\s_*^{(i)}(s,G_2,F_{2,n},\mu_{2,N}) - \s_*^{(i)}(s,G_1,F_{1,n},\mu_{1,N})]\,dW_n^{(i)}(s)\Big| \notag
\\&+\Big|\int_0^t[\s_*^{(c)}(s,G_2,F_{2,n},\mu_{2,N}) - \s_*^{(c)}(s,G_1,F_{1,n},\mu_{1,N})]\,dW^{(c)}(s)\Big|. \label{contraction.estimate.finite.1}
\end{align}
Let us estimate $(I)$. By \eqref{convoluted.operator.lipschitz.pointwise} and by \eqref{Wasserstein.empirical.measure} we get
\begin{align*}|H_1*\mu_{2,N}(F_{2,n}) - H_1*\mu_{1,N}(F_{2,n})|(s)
&\le L\cdot\mW_1(\mu_{2,N}(s),\mu_{1,N}(s))\le L|F_{2}(s) - F_{1}(s)|
\end{align*}
so that
\begin{equation*}(I)\lesssim_L\int_0^t|F_2(s) - F_1(s)|\,ds.
\end{equation*}
As for $(II)$, thanks to \eqref{convoluted.function.lipschitz} we get the estimate
\begin{equation*}(II)\lesssim_L\int_0^t|F_2(s) - F_1(s)|\,ds,
\end{equation*}
and similarly it holds for $(III)$ that
\begin{equation*}(III)\lesssim_L\int_0^t[|F_2(s) - F_1(s)| + |G_2(s) - G_1(s)|]\,ds.
\end{equation*}
Coming back to \eqref{contraction.estimate.finite.1}, we get
\begin{align}|\Phi_1(F_2,G_2)(t)-\Phi_1(F_1,G_1)(t)|
&\lesssim_L\int_0^t[|F_2(s) - F_1(s)| + |G_2(s) - G_1(s)|]\,ds \notag
\\&+\sup_{s\in[0,t]}\Big|\int_0^s[\s_*^{(i)}(z,G_2,F_{2,n},\mu_{2,N}) - \s_*^{(i)}(z,G_1,F_{1,n},\mu_{1,N})]\,dW_n^{(i)}(z)\Big|\notag
\\&+\sup_{s\in[0,t]}\Big|\int_0^s[\s_*^{(c)}(z,G_2,F_{2,n},\mu_{2,N}) - \s_*^{(c)}(z,G_1,F_{1,n},\mu_{1,N})]\,dW^{(c)}(z)\Big|\label{est.prov}
\end{align}
By the same arguments and by using also \eqref{u.lipschitz.in.measure}, we can similarly get the estimate
\begin{equation*}|\Phi_2(F_2,G_2)(t)-\Phi_2(F_1,G_1)(t)|
\lesssim_L\int_0^t[|F_2(s) - F_1(s)| + |G_2(s) - G_1(s)|]\,ds;
\end{equation*}
Let us exponentiate both sides of \eqref{est.prov} by $p$, take the expectation, and we get
\begin{align}\mathbb{E}[|\Phi(F_2,G_2)(s)-\Phi(F_1,G_1)(s)|^p]
&\lesssim_{p,L}\int_0^t\mathbb{E}[|F_2(s) - F_1(s)|^p + |G_2(s) - G_1(s)|^p]\,ds \notag
\\&+\underbrace{\mathbb{E}\Big[\sup_{s\in[0,t]}\Big|\int_0^s[\s_*^{(i)}(z,G_2,F_{2,n},\mu_{2,N}) - \s_*^{(i)}(z,G_1,F_{1,n},\mu_{1,N})]\,dW_n^{(i)}(z)\Big|^p\Big]}_{=:(I)} \notag
\\&+\underbrace{\mathbb{E}\Big[\sup_{s\in[0,t]}\Big|\int_0^s[\s_*^{(c)}(z,G_2,F_{2,n},\mu_{2,N}) - \s_*^{(c)}(z,G_1,F_{1,n},\mu_{1,N})]\,dW^{(c)}(z)\Big|^p\Big]}_{=:(II)}. \label{contraction.estimates.finite.2}
\end{align}
Let us estimate $(I)$. By \eqref{Doob.maximal.inequality}, \eqref{idiosyncratic.noise.is.lipschitz} and \eqref{Wasserstein.empirical.measure} we have
\begin{align*}(I)
&\lesssim_{p,T}\int_0^t\mathbb{E}[|\s_*^{(i)}(s,G_2,F_{2,n},\mu_{2,N}) - \s_*^{(i)}(s,G_1,F_{1,n},\mu_{1,N})|^p]\,ds
\\&\lesssim_{p,T,L}\int_0^t\mathbb{E}[|F_2 - F_1|^p + |G_2 - G_1|^p +\mW_1^p(\mu_{2,N},\mu_{1,N})]\,ds
\\&\lesssim_{p,T,L}\int_0^t\mathbb{E}[|F_2 - F_1|^p + |G_2 - G_1|^p]\,ds;
\end{align*}
similarly, one estimates $(II)$.
Coming back to \eqref{contraction.estimates.finite.2}, we get
\begin{align*}\mathbb{E}[|\Phi(F_2,G_2)(s)-\Phi(F_1,G_1)(s)|^p]
&\lesssim_{p,T,L}\int_0^t\mathbb{E}[|F_2(s) - F_1(s)|^p + |G_2(s) - G_1(s)|^p]\,ds
\\&\lesssim_{p,T,L}\frac{1}{p\g}e^{p\g t}\|(F_2,G_2)-(F_1,G_1)\|_{\g,p}^p:
\end{align*}
if we multiply both sides by $e^{-p\g t}$, take the supremum over $t\in[0,T]$ and exponentiate by $\frac{1}{p}$, we get that for $\g$ big enough, $\Phi$ is a contraction, so we conclude.

We finally remark that since for each $(F,G)$, $\Phi(F,G)$ is a $\mF_t-$adapted process, then also the solution of \eqref{SDE.finite}-\eqref{Initial.data.finite} is by fixed points iterations.

\end{proof}

As far as the mean-field problem \eqref{SDE.mean.field}-\eqref{Initial.data.mean.field} is concerned, its well-posedness in the sense of Definition \ref{Def.conditional.solutions}, point $(ii)$, is proved in the following lemma:
\begin{lemma}[Mean-field problem well posedness]\label{thm:mean.field.well.posedness}
Given Assumptions \ref{assumptions}, for any $p\ge2$ one has that system \eqref{SDE.mean.field}-\eqref{Initial.data.mean.field} has a unique solution in the sense of Definition \ref{Def.conditional.solutions}, point $(ii)$. Moreover, if $(X,Y)$ is the solution of \eqref{SDE.mean.field}-\eqref{Initial.data.mean.field} and $\mu^{(i)}(t)$ is the law of $X(t)$ given $W_t^{(c)}$, one has that
\begin{align}\sup_{t\in[0,T]}\max\{\mathbb{E}[|X(t)|^p],\mathbb{E}[|Y(t)|^p],\mathbb{E}[|M_p^p(\mu^{(i)}(t))|]\}\le C, \label{mean.field.solution.boundedness} 
\end{align}
where $C=C(p,L,T,H_2(0),K_2(0),X^{(0)}, Y^{(0)},u)$.
\end{lemma}

\begin{proof}
We split the proof in two steps.

\medskip

$(i)$ \emph{Fix the conditional law $\nu^{(i)}\in C([0,T];L^p(\Om;\mW_p(\R^d)))$ and solve the approximated problem.}

\medskip

We want to show that the approximated problem
\begin{align}&dX(t)=\Big\{H_1*\nu^{(i)}(t)(X(t)) + \frac{1}{M}\sum_{j=1}^M K_1(Y_j(t) - X(t))\Big\}dt
\\&\qquad\quad+ \Big\{\s_*^{(i)}(t,Y(t),X(t),\nu^{(i)}(t))\Big\}dW^{(i)}(t) + \Big\{\s_*^{(c)}(t,Y(t),X(t),\nu^{(i)}(t))\Big\}dW^{(c)}(t), \label{SDE.mean.field.approximate}
\\&\notag
\\&\frac{d}{dt}Y_m(t)=K_2*\nu^{(i)}(t)(Y_m(t)) + \frac{1}{M}\sum_{j=1}^M H_2(Y_m(t)-Y_j(t)) + u_m(t,Y(t),\nu^{(i)}(t)),\label{ODE.mean.field.approximate}
\\&\notag
\\&X(0)=X^{(0)},\qquad Y_m(0)=Y_m^{(0)},\qquad1\le m\le M\label{Initial.data.mean.field.approximate}
\end{align}
has a solution in the sense of Definition \ref{Def.conditional.solutions}. Really, this is true by arguing in the same way done in Lemma \ref{lemma:finite.well.posedness} but considering the space $C([0,T];L^p(\Om;\R^d)\times(\R^d)^M)$ for $(X,Y)$.

\medskip

$(ii)$ \emph{Find the conditional law $\mu^{(i)}(t)$ of $X(t)$ given $W^{(c)}(t)$.}

\medskip

We now want to find a strong solution of \eqref{SDE.mean.field}-\eqref{ODE.mean.field}. To this term, let us consider the operator $\mS\colon\,(\nu,W,Z^{(0)})\longmapsto S(\nu^{(i)},W,Z^{(0)})$ which takes a random measure $\nu\in C([0,T];L^p(\Om;\mW_p(\R^d)))$, a Brownian motion $W=(W^{(i)},W^{(c)})$ and an initial datum $Z^{(0)}:=(X^{(0)},Y^{(0)})\in L^p(\Om;\R^d)\times(\R^d)^M$, and gives the solution of \eqref{SDE.mean.field.approximate}-\eqref{Initial.data.mean.field.approximate}, namely
\begin{align*}&\mS(\nu,W,Z^{(0)}):=(\mS_0;\mS_1,\dots,\mS_M),
\\&\mS_0:=\mS_0(\nu,W,Z^{(0)})=X_\nu,\qquad\mS_m:=\mS_m(\nu,W,Z^{(0)})=(Y_m)_\nu,\quad1\le m\le M.
\end{align*}
Fixed the Brownian motion $W$ and the initial datum $Z^{(0)}$, if we call $B:=C([0,T];L^p(\Om;\mW_p(\R^d)))$ we can define the operator $\mT\colon\,B\longmapsto B$, such that for all $t\in[0,T]$ and for all $\nu\in B$,
\begin{align}\label{Tau.map.nu}\mT(\nu)(t):=\text{Law}(S_0(\nu,W,Z^{(0)})(t)|\mF_t^{(c)}).
\end{align}
As a result, we get a solution of \eqref{SDE.mean.field}-\eqref{ODE.mean.field} if we find a fixed point for \eqref{Tau.map.nu}. We want to prove that in fact, $\mT$ is a contraction map if we equip $B$ with the distance 
\begin{equation*}d'_{p,\g}(\nu_2,\mu_1):=\sup_{t\in[0,T]}e^{-\g t}(\mathbb{E}^{(c)}[\mW_p^p(\nu_2(s),\nu_1(s))])^{\frac{1}{p}}    \end{equation*}
where $d_{p,\g}$ is defined in \eqref{decay.distance.measures}. Of course, the map $\mT$ is well-posed.

\medskip

Let us fix two measures $\nu_1,\nu_2$. By \eqref{p.Wasserstein.metric}
we have that for all $t\in[0,T]$,
\begin{align}e^{-p\g t}(\mW_p(\mT(\nu_1)(t),\mT(\nu_2)(t)))^p 
&\le\mathbb{E}^{(i)}[e^{-p\g t}|X_{\nu_2}(t)-X_{\nu_1}(t)|^p]. \label{T.contraction.estimate.1}
\end{align}
To proceed with the estimate \eqref{T.contraction.estimate.1}, we need to estimate $\mathbb{E}^{(i)}[e^{-p\g t}|X_{\nu_2}(t)-X_{\nu_1}(t)|^p]$ and we want to do this by applying Proposition \ref{prop:Gronwall.inequality}. By arguing as in the proof of Lemma \ref{lemma:finite.well.posedness}, we firstly observe that for all $m$ with $1\le m\le M$, arguing as done in the proof of Lemma \ref{lemma:finite.well.posedness} we have
\begin{align*}
|(Y_m)_{\nu_2}(t)-(Y_m)_{\nu_1}(t)|
&\le \int_0^t|K_2*\nu_2((Y_m)_{\nu_2}) - K_2*\nu_1((Y_m)_{\nu_1})|\,ds
\\&+\frac{1}{M}\sum_{j=1}^M\int_0^t|H_2((Y_m)_{\nu_2} - (Y_j)_{\nu_2}) - H_2((Y_m)_{\nu_1} - (Y_j)_{\nu_1})|\,ds
\\&+\int_0^t|u_m(s,Y_{\nu_2},\nu_2) - u_m(s,Y_{\nu_1},\nu_1)|\,ds
\\&\lesssim_L\int_0^t|K_2*\nu_2((Y_m)_{\nu_2}) - K_2*\nu_2((Y_m)_{\nu_1})|\,ds
\\&+\int_0^t|K_2*\nu_2((Y_m)_{\nu_1}) - K_2*\nu_1((Y_m)_{\nu_1})|\,ds
\\&+\int_0^t|Y_{\nu_2}(s) - Y_{\nu_1}(s)|\,ds
\\&+\int_0^t|u_m(s,Y_{\nu_2},\nu_2) - u_m(s,Y_{\nu_1},\nu_1)|\,ds
\\&\lesssim_{L}e^{\g t}\int_0^t\sup_{z\in[0,s]}e^{-\g s}|Y_{\nu_2}(z) - Y_{\nu_1}(z)|\,ds + \frac{1}{\g}e^{\g t}d_{p,\g}(\nu_2,\nu_1);
\end{align*}
by Proposition \ref{prop:Gronwall.inequality}, we get that for all $t\in[0,T]$
\begin{equation*}\sup_{s\in[0,t]}e^{-\g s}|Y_{\nu_2}(s) - Y_{\nu_1}(s)|\lesssim_{L,T} \frac{1}{\g}d_{p,\g}(\nu_2,\nu_1).\end{equation*}
Thanks to the latter inequality, similarly
we get that for all $t\in[0,T]$,
\begin{align}
|X_{\nu_2}(t)-X_{\nu_1}(t)|
&\le\int_0^t|H_1*\nu_2(X_{\nu_2}) - H_1*\nu_1(X_{\nu_1})|\,ds \notag
\\&+\frac{1}{M}\sum_{m=1}^M\int_0^t|K_1(X_{\nu_2}-(Y_m)_{\nu_2})- K_1(X_{\nu_1}-(Y_m)_{\nu_1})|\,ds \notag
\\&+\Big|\int_0^t[\s_*^{(i)}(s,Y_{\nu_2},X_{\nu_2},\nu_2) - \s_*^{(i)}(s,Y_{\nu_1},X_{\nu_1},\nu_1)]\,dW^{(i)}(s) \Big| \notag
\\&+\Big|\int_0^t[\s_*^{(c)}(s,Y_{\nu_2},X_{\nu_2},\nu_2) - \s_*^{(c)}(s,Y_{\nu_1},X_{\nu_1},\nu_1)]\,dW^{(c)}(s) \Big| \notag
\\&\lesssim_L\int_0^t|H_1*\nu_2(X_{\nu_2}) - H_1*\nu_2(X_{\nu_1})|\,ds \notag\\&+\int_0^t|H_1*\nu_2(X_{\nu_1}) - H_1*\nu_1(X_{\nu_1})|\,ds \notag
\\&+\int_0^t[|X_{\nu_2}-X_{\nu_1}| + |Y_{\nu_2} - Y_{\nu_1}|] \,ds\notag
\\&+\sup_{s\in[0,t]}\Big|\int_0^s[\s_*^{(i)}(z,Y_{\nu_2},X_{\nu_2},\nu_2) - \s_*^{(i)}(z,Y_{\nu_1},X_{\nu_1},\nu_1)]\,dW^{(i)}(z) \Big| \notag
\\&+\sup_{s\in[0,t]}\Big|\int_0^s[\s_*^{(c)}(z,Y_{\nu_2},X_{\nu_2},\nu_2) - \s_*^{(c)}(z,Y_{\nu_1},X_{\nu_1},\nu_1)]\,dW^{(c)}(z) \Big| \notag
\\&\lesssim_{C}e^{\g t}\int_0^t\sup_{s\in[0,z]} e^{-\g z}|X_{\nu_2} - X_{\nu_1}|\,ds + e^{\g t}d_{1,\g}(\nu_2,\nu_1) \notag
\\&+\sup_{s\in[0,t]}\Big|\int_0^t[\s_*(s,Y_{\nu_2},X_{\nu_2},\nu_2) - \s_*(s,Y_{\nu_1},X_{\nu_1},\nu_1)]\,dW(s) \Big| \notag
\\&+\sup_{s\in[0,t]}\Big|\int_0^s[\s_*^{(c)}(z,Y_{\nu_2},X_{\nu_2},\nu_2) - \s_*^{(c)}(z,Y_{\nu_1},X_{\nu_1},\nu_1)]\,dW^{(c)}(z) \Big|,
\label{T.contraction.estimate.2}
\end{align}
where $C=C(p,L,T,H_i(0),K_i(0))$, with $i\in\{1,2\}$.
If we exponentiate both sides by $p$ and take the expectation, we get again in the same way as in the proof of \ref{lemma:finite.well.posedness}
\begin{align*}\mathbb{E}[e^{-p\g s}|X_{\nu_2}(s)-X_{\nu_1}(s)|^p]
&\lesssim_{C}\int_0^t\sup_{z\in[0,s]}\mathbb{E}[e^{-p\g t}|X_{\nu_2}(z)-X_{\nu_1}(z)|^p]\,ds + \frac{1}{\g^p}(d_{p,\g}')^p(\nu_2,\nu_1),
\end{align*}    
so by Proposition \ref{prop:Gronwall.inequality} we have
\begin{equation*}\mathbb{E}[e^{-p\g s}|X_{\nu_2}(s)-X_{\nu_1}(s)|^p]
\lesssim_{C}\frac{1}{\g^p}(d_{p,\g}')^p(\nu_2,\nu_1),\end{equation*}
Thanks to the latter inequality, by taking the expectation $\mathbb{E}^{(c)}$ of both sides of \eqref{T.contraction.estimate.1} and exponentiating by $\frac{1}{p}$ we get
\begin{align*}d_{p,\g}'(\mT(\nu_2),\mT(\nu_1))\lesssim_{C}\frac{1}{\g}\cdot d_{p,\g}'(\nu_2,\nu_1);
\end{align*}
we finally get that for $\g$ big enough, $\mT$ is a contraction, so we conclude.

\end{proof}

\begin{remark}\label{rem:conditional.independence} In the absence of common noise, it is obvious that given $\mF-$independent initial data $(X_0^{(1)},Y_0^{(1)})$ and $(X_0^{(2)},Y_0^{(2)})$, the corresponding solutions are still independent for all times. In particular, the stochastic paths $X_0^{(1)}(t),\,X_0^{(2)}(t)$ are independent for each time. When there is the common noise, we must specify what we mean by independence. 

Intuitively, we would get independence if we start by $\mF-$independent initial data and we fix the realization of common noise. This can be formalized in terms of conditional independence given $\mF_t^{(c)}$: we say that $X^{(1)}(t)$ and $X^{(2)}(t)$ are conditionally independent given $\mF_t^{(c)}$ if for any bounded Borel-measurable functions $\ph_1,\,\ph_2\colon\R^d\longmapsto\R$, $\mF_t^{(c)}-$almost surely,
\begin{align*}&\mathbb{E}\Big[\ph_1(X^{(1)}(t))\cdot\ph_2(X^{(2)}(t))\Big|\,\mF_t^{(c)}\Big]
=\mathbb{E}\Big[\ph_1(X^{(1)}(t))\Big|\,\mF_t^{(c)}\Big]\cdot\mathbb{E}\Big[\ph_2(X^{(2)}(t))\Big|\,\mF_t^{(c)}\Big].
\end{align*}
Now, since we have found a unique solution $(X,Y)$ to \eqref{SDE.mean.field}-\eqref{Initial.data.mean.field}, by Remark \ref{rem:conditional.law} we get that starting from independent, identically distributed initial data, these two properties are preserved under the action of the random conditional measure flow.

\end{remark}

\section{The mean-field limit and the (conditional) propagation of chaos}

Now, we have all the tools to show that the discrete problem converge to the mean-field one:

\begin{theorem}[Mean-field limit]\label{thm:propagation.of.chaos}
Given Assumptions \ref{assumptions}, let $p\ge2$, $(W_n^{(i)})_{n\in\N}$ be a sequence respectively of independent idiosyncratic Brownian motions $W^{(c)}$ be the common Brownian motion being independent by each $W_n^{(i)}$, and let 
\begin{equation*}(X_n^{(0)})_{n\in\N}\subset L^p(\Om;\R^d),\,(Y^{(0)})\in(\R^d)^M\end{equation*}
be initial data such that they are independent, identically distributed random variables. Let $(X^{(N)},Y^{(N)})$ be the unique solution of the system \eqref{SDE.finite}-\eqref{Initial.data.finite} with initial data 
\begin{equation*}(X^{(0,N)},Y^{(0,N)})=(X_{n}^{(0)},Y^{(0)})_{1\le n\le N},\end{equation*}
and let $(X_n,Y)$ be the unique solution to the system \eqref{SDE.mean.field}-\eqref{Initial.data.mean.field} with initial data $(X_n^{(0)},Y^{(0)})$. Let $\mu^{(i)}:=\mu_n^{(i)}\in C([0,T];L^p(\Om;\mW_p(\R^d)))$ be the conditional law of each $X_n$ given $(\mF_t^{(c)})_{t\in[0,T]}$.
Then, the following two facts hold.

\medskip

$(i)$ For all $q\in[1,p)$ as in \eqref{law.of.large.numbers},
\begin{align}&\mathbb{E}\Big[\Big(\max_{1\le n\le N}\sup_{t\in[0,T]}|X_n(t)-X_n^{(N)}(t)|^q + \sup_{t\in[0,T]}|Y(t)-Y^{(N)}(t)|^q\Big)\Big]\lesssim_{q,p,L,T,\mu^{(i)},u}C(N,p,q),\label{propagation.of.chaos}
\\&\mathbb{E}[\mW_q^q(\mu^{(i)}(s),\mu_N(s))]\lesssim_{q,p,L,T,\mu^{(i)},u}C(N,p,q),\label{convergence.of.measures}
\end{align}
where $C(N,p,q)$ is that of Proposition \ref{prop:technical.estimates}.

$(ii)$ For $q=p$,
\begin{align}&\lim_{N\to\infty}\mathbb{E}\Big[\Big(\max_{1\le n\le N}\sup_{t\in[0,T]}|X_n(t)-X_n^{(N)}(t)|^p + \sup_{t\in[0,T]}|Y(t)-Y^{(N)}(t)|^p\Big)\Big]=0,\label{propagation.of.chaos.1}
\\&\lim_{N\to\infty}\mathbb{E}[\mW_p^p(\mu^{(i)}(s),\mu_N(s))]=0.\label{convergence.of.measures.1}
\end{align}
    
\end{theorem}

\begin{proof}

$(i)$ Let us consider $\displaystyle q\in[1,p)$ first, where $q$ satisfies the cases \eqref{law.of.large.numbers}.
Let us fix $n,N\in\N$ and $t\in[0,T]$, and let us call $\mu_N^{(i)}$ the empirical measure associated to $(X_n)_{1\le n\le N}$. 
By Definition \ref{Def.conditional.solutions}, neglecting dependence by time and $\om\in\Om$ for simplicity of notations, we have
\begin{align}X_n(t)-X_n^{(N)}(t)
&=\int_0^t[H_1*\mu_n^{(i)}(X_n) - H_1*\mu_N(X_n^{(N)})]\,ds \notag
\\&+\frac{1}{M}\sum_{m=1}^M\int_0^t[K_1(Y_m-X_n) - K_1(Y_m^{(N)} - X_n^{(N)})]\,ds \notag
\\&+\int_0^t[\s_*^{(i)}(s,Y,X_n,\mu_n^{(i)}) - \s_*^{(i)}(s,Y^{(N)},X_n^{(N)},\mu_N)]\,dW_n^{(i)}(s) \notag
\\&+\int_0^t[\s_*^{(c)}(s,Y,X_n,\mu_n^{(i)}) - \s_*^{(c)}(s,Y^{(N)},X_n^{(N)},\mu_N)]\,dW^{(c)}(s)
\label{x}
\end{align}
and
\begin{align}Y_m(t)-Y_m^{(N)}(t)
&=\int_0^t[K_2*\mu_n^{(i)}(Y_m) - K_2*\mu_N(Y_m^{(N)})]\,ds \notag
\\&+\frac{1}{M}\sum_{j=1}^M\int_0^t[H_2(Y_m-Y_j) - H_2(Y_m^{(M)} - Y_j^{(N)})]\,ds\notag
\\& +\int_0^t[u_m(s,Y,\mu_n^{(i)}) - u_m(s,Y^{(N)},\mu_N)]\,ds, \label{y}
\end{align}
Let us start by estimating \eqref{x}. By \eqref{K.L.Lipschitz}, we get
\begin{align*}|X_n(t)-X_n^{(N)}(t)|^q
&\lesssim_{q,L,T} \underbrace{\int_0^t|H_1*(\mu_N^{(i)}-\mu_N)(X_n^{(N)})|^q\,ds}_{=:(I)}
\\&+\underbrace{\int_0^t|H_1*\mu_N^{(i)}(X_n^{(N)}) - H_1*\mu_N^{(i)}(X_n)|^q\,ds}_{=:(II)}
\\&+\underbrace{\int_0^t|H_1*(\mu_N^{(i)}-\mu_n^{(i)})(X_n)|^q\,ds}_{=:(III)}
\\&+\int_0^t\Big[\max_{1\le n\le N}\sup_{z\in[0,s]}|X_n(z)-X_n^{(N)}(z)|^q + \sup_{z\in[0,s]}|Y(z)-Y^{(N)}(z)|^q \Big]\,ds\Big\}
\\&+\sup_{s\in[0,t]}\Big|\int_0^s[\s_*^{(i)}(z,Y,X_n,\mu_n^{(i)}) - \s_*^{(i)}(z,Y^{(N)},X_n^{(N)},\mu_N)]\,dW_n^{(i)}(z)\Big|^q
\\&+\sup_{s\in[0,t]}\Big|\int_0^s[\s_*^{(c)}(z,Y,X_n,\mu_n^{(i)}) - \s_*^{(c)}(z,Y^{(N)},X_n^{(N)},\mu_N)]\,dW^{(c)}(z)\Big|^q.
\end{align*}
The term $(II)$ can be immediately estimated by using \eqref{convoluted.function.lipschitz}:
\begin{equation*}(II)\lesssim_{q,L}\int_0^t\max_{1\le n\le N}\sup_{z\in[0,s]}|X_n(z)-X_n^{(N)}(z)|^q\,ds.
\end{equation*}
The term $(I)$ can be estimated $\mathbb{P}^{(c)}-$almost surely by \eqref{K.L.Lipschitz}, \eqref{convoluted.operator.lipschitz.pointwise} and \eqref{Wasserstein.empirical.measure}:
\begin{equation*}(I)\lesssim_{q,L}\int_0^t\mW_1^q(\mu_N^{(i)}(s),\mu_N(s))\,ds\lesssim_{q,L}\int_0^t\max_{1\le n\le N}\sup_{z\in[0,s]}|X_n(z)-X_n^{(N)}(z)|^q\,ds;
\end{equation*}
similarly,
\begin{equation*}(III)\lesssim_{q,L}\int_0^t\mW_1^q(\mu_N^{(i)}(s),\mu_n^{(i)}(s))\,ds.
\end{equation*}
Putting the estimates for $(I), (II), (III)$ together, we get
\begin{align*}\max_{1\le n\le N}\sup_{s\in[0,t]}|X_n(s)-X_n^{(N)}(s)|^q
&\lesssim_{q,L,T} \int_0^t\max_{1\le n\le N}\sup_{z\in[0,s]}|X_n(z)-X_n^{(N)}(z)|^q\,ds
\\&+ \int_0^t\sup_{z\in[0,s]}|Y(z)-Y^{(N)}(z)|^q\,ds
\\&+\int_0^t\mW_1^q(\mu_N^{(i)}(s),\mu_n^{(i)}(s))\,ds
\\&+\sup_{s\in[0,t]}\Big|\int_0^s[\s_*^{(i)}(z,Y,X_n,\mu_n^{(i)}) - \s_*^{(i)}(z,Y^{(N)},X_n^{(N)},\mu_N)]\,dW_n^{(i)}(z)\Big|^q
\\&+\sup_{s\in[0,t]}\Big|\int_0^s[\s_*^{(c)}(z,Y,X_n,\mu_n^{(i)}) - \s_*^{(c)}(z,Y^{(N)},X_n^{(N)},\mu_N)]\,dW^{(c)}(z)\Big|^q.
\end{align*}
Now, let us estimate \eqref{y}. Arguing as done for \eqref{x}, we have
\begin{align*}|Y_{m}(t)-Y_{m}^{(N)}(t)|^q
&\lesssim_{q,T}\int_0^t|K_2*\mu_n^{(i)}(Y_m) - K_2*\mu_N(Y_m^{(N)})|^q\,ds
\\&+\frac{1}{M}\sum_{j=1}^M\int_0^t|H_2(Y_m-Y_j) - H_2(Y_m^{(N)} - Y_j^{(N)})|^q\,ds
\\&+\int_0^t|u_m(s,Y,\mu_n^{(i)}) - u_m(s,Y^{(N)},\mu_N)|^q\,ds
\\&\lesssim_{q,T,L}\int_0^t\sup_{z\in[0,s]}|Y(z)-Y^{(N)}(z)|^q\,ds
\\&+\int_0^t\mW_1^q(\mu_n^{(i)}(s),\mu_N^{(i)}(s))\,ds 
\\&+\int_0^t \max_{1\le n\le N}\sup_{z\in[0,s]}|X_n(z)-X_n^{(N)}(z)|^q\,ds\Big\}
\end{align*}
In conclusion, calling
\begin{equation*}G_N(t):=\max_{1\le n\le N}\sup_{s\in[0,t]}|X_n(s)-X_n^{(N)}(s)|^q + \sup_{s\in[0,t]}|Y(s)-Y^{(N)}(s)|^q,
\end{equation*}
we have
\begin{align*}G_N(t)&\lesssim_{q,L,T}\int_0^t G_N(s)\,ds + \int_0^t\mW_1^q(\mu_n^{(i)}(s),\mu_N^{(i)}(s))\,ds
\\&+\sup_{s\in[0,t]}\Big|\int_0^s[\s_*^{(i)}(z,Y,X_n,\mu_n^{(i)}) - \s_*^{(i)}(z,Y^{(N)},X_n^{(N)},\mu_N)]\,dW_n^{(i)}(z)\Big|^q
\\&+\sup_{s\in[0,t]}\Big|\int_0^s[\s_*^{(c)}(z,Y,X_n,\mu_n^{(i)}) - \s_*^{(c)}(z,Y^{(N)},X_n^{(N)},\mu_N)]\,dW^{(c)}(z)\Big|^q;
\end{align*}
taking the expectation, as usual we have
\begin{align*}\mathbb{E}[G_N(t)]
&\lesssim_{q,L,T}\int_0^t \mathbb{E}[G_N(s)]\,ds + \int_0^t\mathbb{E}[\mW_1^q(\mu_n^{(i)}(s),\mu_N^{(i)}(s))]\,ds
\end{align*}
and calling $v_N(t):=\mathbb{E}[G_N(t)]$ and $I_N:=\int_0^T\mathbb{E}[\mW_1^q(\mu_n^{(i)}(s),\mu_N^{(i)}(s))]\,ds$, by Proposition \eqref{prop:Gronwall.inequality} we get that for all $t\in[0,T]$,
\begin{equation}\label{v}v_N(t)\lesssim_{q,L,T}I_N.
\end{equation}
We now want to estimate $I_N$. Since $\displaystyle q\in[1,p)$, by \eqref{Wasserstein.sort} and \eqref{law.of.large.numbers} we have that $\mathbb{P}-$almost surely,
\begin{equation*}
\mathbb{E}[\mW_1^q(\mu_n^{(i)}(s),\mu_N^{(i)}(s))]\lesssim_{p,q}M_p^p(\mu_n^{(i)}(s))\cdot C(N,p,q),
\end{equation*}
where $M_p^p(\mu_n^{(i)})$ satisfies the boundedness estimate \eqref{mean.field.solution.boundedness} while $C(N,p,q)\to0$ as $N\to+\infty$, and so by Lebesgue Dominated Convergence Theorem we finally get \eqref{propagation.of.chaos}.

\medskip

$(ii)$ The estimate \eqref{v} converges to $0$ by reasoning as in Remark \ref{remark:non.quantitative.law.large.numbers}; the fact that $I_N\to0$ follows by boundedness estimates \eqref{mean.field.solution.boundedness} (which gives uniform estimates in time), and so one can apply Lebesgue Dominated Convergence Theorem.

\end{proof}

\begin{remark}\label{remark} As for quantitative estimates \eqref{propagation.of.chaos} and \eqref{convergence.of.measures}, in all the cases $q\in[1,p)$ not covered by Proposition \ref{prop:technical.estimates} can be obtained through Jensen Inequality. 
\end{remark}

\begin{remark} We notice that since $\frac{d}{2}\in\{\frac12,1,\frac32\}$, when $q=1$ we have by Proposition \ref{prop:technical.estimates} that if $d=1$, then $I_N\lesssim_{p} N^{-\frac12}$; if $d=3$, then $I_N\lesssim_{p}N^{-\frac{1}{3}}$; if else $d=2$, then one has that for $p>2$, $I_N\lesssim_p N^{-\frac12}\log N$, while for $p=2$, using Remark \ref{remark} one has that for any $q>1$, $I_N\lesssim N^{-\frac{1}{2q}}$. In this way, we recover the necessary estimate for the mean-field limit.

\end{remark}

The mean-field limit theorem we have just proved has the following striking consequence: the propagation of chaos if there is no common noise, otherwise the conditional propagation of chaos.

\begin{corollary}[(Conditional) Propagation of chaos]\label{cor:propagation.of.chaos}

In the notations and hypotheses of Theorem \ref{thm:propagation.of.chaos}, let us consider $k\ge2$ observables $(\ph_i)_{1\le i\le k}\subset C_b(\R^d)\cap\text{Lip}(\R^d)$. Then, for all $t\in[0,T]$ and for all set of indices $(n_i)_{1\le i\le k}$,
\begin{align}&\lim_{N\to+\infty}\mathbb{E}\Big[\prod_{i=1}^k\ph_i(X_{n_i}^{(N)}(t))\Big]=\prod_{i=1}^k\mathbb{E}[\ph_i(X(t))]=\prod_{i=1}^k\la\mu(t),\ph_i\ra \qquad\qquad\qquad\quad\quad\text{if}\,\,\s_*^{(c)}\equiv0,\label{propagation.of.chaos.equation}
\\&\lim_{N\to+\infty}\mathbb{E}\Big[\prod_{j=1}^k\ph_j(X_{n_j}^{(N)}(t))\Big|\mF_t^{(c)}\Big]=\prod_{j=1}^k\mathbb{E}[\ph_j(X(t))|\mF_t^{(c)}]=\prod_{j=1}^k\la\mu^{(i)}(t),\ph_j\ra\qquad\,\,\text{otherwise.}\label{propagation.of.chaos.equation.common}
\end{align}

\end{corollary}

\begin{proof}
Let us suppose that $\s_*^{(c)}\equiv0$. By Remark \ref{remark:no.common.noise.gives.deterministic.law}, the law $\mu$ for each $X_n$ is deterministic.

Let then $k=2$. Without any loss of generality, let us suppose $n_1:=1,\,n_2:=2$, let $B$ be the constant bounding both $\ph_1,\,\ph_2$ and let $L$ their Lipschitz constant. Then, since $X_1(t)$ and $X_2(t)$ are independent by the fact that their initial data are independent, we have
\begin{align*}&\mathbb{E}[\ph_1(X_1^{(N)}(t))\ph_2(X_2^{(N)}(t))] - \mathbb{E}[\ph_1(X_1(t))]\cdot\mathbb{E}[\ph_2(X_2(t))]
\\&\qquad\qquad=\mathbb{E}[\ph_1(X_1^{(N)}(t))\ph_2(X_2^{(N)}(t)) - \ph_1(X_1(t))\ph_2(X_2(t))]
\\&\qquad\qquad\le\mathbb{E}[|\ph_1(X_1^{(N)}(t))|\cdot|\ph_2(X_2^{(N)}(t)) - \ph_2(X_2(t))|]
\\&\qquad\qquad\quad+\mathbb{E}[|\ph_2(X_2(t))|\cdot|\ph_1(X_1^{(N)}(t)) - \ph_1(X_1(t))|]
\\&\qquad\qquad\lesssim_{B,L}\cdot\mathbb{E}[|X_1^{(N)}(t) - X_1(t)| + |X_2^{(N)}(t) - X_2(t)|],
\end{align*}
which tends to $0$ by Theorem \ref{thm:propagation.of.chaos}. For $k>2$, one can proceed in the same way as in the case $k=2$. 

Let us now suppose that $\s_*^{(c)}\not\equiv0$.
Fixed any $\om\in\Om$, we have by definition of strong solutions that $\mathbb{E}[\ph_1(X(t))|\mF_t^{(c)}]$ and $\mathbb{E}[\ph_2(X(t))|\mF_t^{(c)}]$ are conditionally $\mF_t^{(c)}-$independent given $\mF_t^{(c)}$ by Remark \ref{rem:conditional.independence}, and so we can repeat the same computation above by replacing the expectation operator $\mathbb{E}$ with $\mathbb{E}_A:=\mathbb{E}[\mathbf{1}_A\cdot]$, with $A\in\mF_t^{(c)}$, applied to the conditional expectations.

\end{proof}

\section{The Fokker-Plack equation for the mean-field problem}

In Theorem \ref{thm:propagation.of.chaos}, we have proved that the herd statistical distribution converges to the law of the mean-field process in the sense of \eqref{convergence.of.measures}. We want to prove that the measure associated to the solution of \eqref{SDE.mean.field}-\eqref{ODE.mean.field} can be made explicit, and it is indeed a measure-valued solution of the associated Fokker-Planck equation. 

In this section, we will always consider notations and assumptions of Theorem \ref{thm:mean.field.well.posedness}, and we also introduce the notations
\begin{align}&V(t,\mu,x):=H_1*\mu(t)(x) + \frac{1}{M}\sum_{m=1}^M K_1(Y_m(t)-x),\label{V.def}
\\&\s_*:=\s_*^{(i)}(\s_*^{(i)})^T + \s_*^{(c)}(\s_*^{(c)})^T\label{sigma.star.def}
\end{align}
We can establish the Fokker-Planck equations for the mean-field problem \eqref{SDE.mean.field}-\eqref{Initial.data.mean.field}:

\begin{theorem}[Fokker-Planck equations for the mean-field problem]\label{thm:Fokker.Planck.via.Ito}
In notations and assumptions of Lemma \ref{thm:mean.field.well.posedness}, we have that $\mathbb{P}_0-$almost surely, $\mu$ solves \eqref{parabolic.equation.common}-\eqref{parabolic.initial.data.common}, in the sense that for all $\ph\in C_c^{\infty}(\R^d)$, $\mathbb{P}^{(0)}-$almost surely
\begin{align}\la\mu^{(i)}(t),\ph\ra-\la\mu_0^{(i)},\ph\ra &= \int_0^t\la\mu^{(i)}(s),\grad\ph\cdot V(s,\mu^{(i)}(s),\cdot) + \frac12\text{tr}\{H(\ph)\s_*(s,\mu^{(i)}(s),\cdot)\}\ra\,ds \notag
\\&+\int_0^t\la\mu^{(i)}(s),\grad\ph\cdot\s_*^{(c)}(s,\mu^{(i)}(s),\cdot)\,dW^{(c)}(s)\ra,\label{duality.fokker.planck.common}
\end{align}
where $\la\mu^{(i)}(t),\ph\ra$ is $\mathbb{P}-$almost surely the average of $\ph$ over $\Om$ given $\mF_t^{(c)}$.

In particular, if $\s_*^{(c)}\equiv0$, then the Fokker-Planck equation reduces to the deterministic PDE Cauchy Problem \eqref{Deterministic.Fokker.Planck.PDE}-\eqref{Deterministic.Fokker.Planck.Initial.datum}.

\end{theorem}

\begin{proof}
We will prove only the existence of measure-valued solutions for \eqref{parabolic.equation.common}-\eqref{parabolic.initial.data.common}, since the zero-common noise is a special case.

Let us take any $\ph\in C_b^2(\R^d)$. Given the unique strong solution $(X,Y)$ for \eqref{SDE.mean.field}-\eqref{Initial.data.mean.field} and the process $\ph(X)$, by \eqref{Ito.formula} its \^Ito differential is
\begin{align*}d\ph(X)&=\Big\{\frac12\text{tr}\{H(\ph)[\s_*^{(i)}(\s_*^{(i)})^T + \s_*^{(c)}(\s_*^{(c)})^T]\}(X) + \grad\ph(X)\cdot V\Big\}dt
\\&+\grad\ph(X)\cdot\s_*^{(i)}dW^{(i)} + \grad\ph(X)\cdot\s_*^{(c)}dW^{(c)},
\end{align*}
integrating both sides, we get
\begin{align*}\ph(X(t))-\ph(X^{(0)}) 
&= \int_0^t\Big\{\frac12\text{tr}\{H(\ph)[\s_*^{(i)}(\s_*^{(i)})^T + \s_*^{(c)}(\s_*^{(c)})^T]\}(X)+ \grad\ph(X)\cdot V \Big\}ds
\\&+\int_0^t \grad\ph(X)\cdot\s_*^{(i)}dW^{(i)}(s) + \int_0^t \grad\ph(X)\cdot\s_*^{(c)}dW^{(c)}(s);
\end{align*}
applying the conditional expectation operator $\mathbb{E}[\cdot|\mF_t^{(c)}]$ and observing that the stochastic integral in $dW^{(i)}$ in the latter identity is a zero-average martingale, by Fubini Theorem we get \eqref{duality.fokker.planck.common}.

\end{proof}

Let us now focus on the pure idiosyncratic noise case $\s_*^{(c)}\equiv0$. In this setting, we simplify the notation $\s_*=\s_*^{(i)}(\s_*^{(i)})^T$, and the Feynman-Kac Formula holds:
\begin{lemma}[Feynman-Kac formula]\label{lemma:Feynman.Kac}

Given Assumptions \ref{assumptions} and any fixed $T>0$, let us fix a measure-valued function $\mu\in C([0,T];\mW_p(\R^d))$ and let us consider the Kolmogorov Backward Problem
\begin{align}&\pa_t u + \frac12\text{tr}(\s_*H(u))+ V\cdot\grad u=0,\qquad\text{in}\quad[0,T)\times\R^d, \label{Parabolic.problem}
\\& u|_{t=T}=\ph\in C_b^2(\R^d), \label{Parabolic.initial.datum}
\end{align}
where $C_b^2(\R^d)$ is the space of the $C^2$ functions which are bounded in $\R^d$. Denoting by $C^{1,2}([0,T)\times\R^d)$ the space of functions being $C^1$ in time and $C^2$ in space, there exists a solution $u\in C([0,T]\times\R^d)\cap C^{2,1}([0,T)\times\R^d)$ given by
\begin{equation}\label{feynman.kac}u(t,x)=\mathbb{E}^{(x,t)}[\ph(\xi_T)],
\end{equation}
where $\mathbb{E}^{(x,t)}$ is the expectation operator which is associated to the diffusion process $\xi$ of starting point $(x,t)\in\R^d\times[0,T]$ and generated by the operator $\mathcal{L}:=\frac12\text{tr}(\s_*H) + V\cdot\grad$; precisely, $\xi$ is the unique solution to the SDE
\begin{align}&d\xi(t)=V(t,\mu(t),\xi(t))dt + \s_*^{(i)}(t,\mu(t),\xi(t))dW(t), \notag
\\&\xi(0)=x \label{sde.markov}
\end{align}

\end{lemma}

\begin{proof} It is a special case of Theorem 2.7, \cite{Pavliotis}, which follows from Assumptions \ref{assumptions}. Indeed, we notice that by Assumptions \ref{assumptions}, $V$ is Lipschitz-continuous in $x,\mu$ uniformly in $t$, and it is continuous in time, and the same holds for $\s_*^{(i)}$. From this, it follows that the solution $\xi$ to \eqref{sde.markov} exists and is unique: this is enough to guarantee that $\xi$ is a diffusion process, and so Feynman-Kac formula \eqref{feynman.kac} is available.
\end{proof}

Thanks to Feynman-Kac, we easily get the uniqueness of measure-valued solutions for \eqref{Parabolic.problem}-\eqref{Parabolic.initial.datum}:

\begin{corollary}\label{cor:uniqueness.fokker.planck.zero.common.noise} In notations and assumptions of Theorem \ref{thm:Fokker.Planck.via.Ito}, for $\s_*^{(c)}\equiv0$ we have that the law $\mu$ of the mean-field process is the unique measure-valued solution of \eqref{Parabolic.problem}-\eqref{Parabolic.initial.datum}.
\end{corollary}

\begin{proof}

Let us observe that one can repeat the computations done in the proof of Theorem \ref{thm:Fokker.Planck.via.Ito} to show that for all $u\in C^{1,2}([0,T)\times\R^d)$, one has that for all $0\le t'\le t<T$,
\begin{align}\la\mu(t),u(t,\cdot)\ra - \la\mu(t'),u(t',\cdot)\ra &= \int_{t'}^t\la\mu(s),\pa_s u(s,\cdot)\ra\,ds + \int_{t'}^t\la\mu(s),V(s,\mu(s),\cdot)\cdot\grad u(s,\cdot)\ra\,ds \notag
\\&+ \frac12\int_{t'}^t\la\mu(s),\text{tr}\{H(u)\s_*(s,\mu^{(i)}(s),\cdot)   \ra\,ds.\label{Fokker.Planck.time.extended}
\end{align}

Let us fix any $\ph\in C_b^2(\R^d)$, and let us consider the solution $u(t,x)$ for the Kolmogorov problem \eqref{Parabolic.problem}-\eqref{Parabolic.initial.datum} with final datum $\ph$ given by \eqref{feynman.kac}. Then, by \eqref{Fokker.Planck.time.extended} one gets that for all $0\le t'\le t<T$, 
\begin{equation*}\la\mu(t),u(t,\cdot)\ra=\la\mu(t'),u(t',\cdot)\ra,\end{equation*}
and as a result, as $t\to T$ and $t'\to0$, by continuity of $\mu$ we get 
\begin{equation*}\la\mu_{0},u(0,\cdot)\ra=\la\mu(T),\ph\ra:\end{equation*}
by arbitrariness of $T$, we have the uniqueness of $\mu$ as well.

\end{proof}

\section{Optimal control for discrete and mean-field problems}

In this section, we will make the following assumptions.

\begin{ass}\label{assumptions.control}

$(i)$ (Stochastic control functions) Fixed an integer $\ell\in\N$ and two positive real numbers $M',L>0$, let us call $\mU\subset\R^{d\times\ell}$ and $\mathcal{G}\subset C((\R^d)^M\times\mW_1(\R^d);\R^{\ell})$ two sets such that:
\begin{align}&\mU,\mG\,\,\text{are compact}, \label{U.G.compatti}
\\&\mG\,\,\text{is made of $M'$-bounded, $L-$Lipschitz-continuous functions}. \label{G.description}
\end{align}
Let us call $E:= L^1([0,T];\mU)\times\mG$, and for any $(h,g)\in\mM(\Om;E^M)$ such that $h$ is a $\mF_t-$adapted process, let us suppose that for any $m\in\N$ such that $1\le m\le M$, and for any $t\in[0,T],\,Y\in C([0,T];(\R^d)^M),\,\nu\in C([0,T];\mW_1(\R^d))$, we have that each $u_m$ satisfies \eqref{separation of variable}.

\medskip

$(ii)$ (Cost functions) In the notations above, let us consider the running cost function $\Psi_\rho\colon\,\R^{d\times\ell}\times\R^{\ell}\longmapsto\R$, the transient cost function $\Psi_\tau\colon\,[0,T]\times(\R^d)^M\times\mW_1(\R^d)\longmapsto\R$ and the endpoint cost function $\Psi_\e\colon\,(\R^d)^M\times\mW_1(\R^d)\longmapsto\R$ such that
\begin{align}
&\Psi_\rho\,\,\text{is continuous in all the variables}\,\, \text{and convex in the first}, \label{Psi.in.assumptions}
\\&\Psi_\t\,\,\text{is uniformly continuous,}  \label{L.assumptions}
\\&\Psi_\e\,\,\text{is uniformly continuous}  \label{e.assumptions}
\end{align}
\end{ass}

\begin{remark}
We point out that the Assumptions \ref{assumptions.control}, point $(i)$, are compatible with those made in Assumptions \ref{assumptions}, point $(iii)$, for the control functions. For $h,g$ there is no need to make any further summability assumptions over $\Om$, since by \eqref{U.G.compatti} and \eqref{G.description} they are both bounded functions, and so $u$ is automatically in $L^p(\Om)$.

Moreover, if $\s_*^{(c)}\equiv0$, then we can assume the latter to be deterministic, and not stochastic, as in \cite{ACS}.

\end{remark}

\begin{remark} We point out some technical assumptions for the cost functions.

The assumption \eqref{Psi.in.assumptions} is expected to be the tightest one for having $\Gamma-$convergence, since the convexity in time variable guarantees the lower semicontinuity of the integral functional $\int_0^T\Psi_\rho\,dt$ with respect to the weak $L^1-$convergence.

The assumptions \eqref{L.assumptions} and \eqref{e.assumptions} are made \emph{ad hoc} for the functional \eqref{F.N.finite} for the discrete problem \eqref{SDE.finite}-\eqref{Initial.data.finite}, since the uniform continuity guarantees the existence of a concave modulus of continuity, so that when the expectation operator $\mathbb{E}$ is applied to the integral functional $\int_0^T\Psi_\t\,dt$ and to $\Psi_\e$, by Jensen Inequality we can get very good estimates which eventually converge to $0$ (see proof of Lemmas \ref{lemma:F.N.has.solution}, \ref{lemma:F.N.has.solution} and Theorem \ref{thm:Gamma.convergence}).

\end{remark}

\begin{remark}

In the assumption \eqref{G.description}, we have required that each $g_m$ in \eqref{separation of variable} has the same $L^\infty-$bound $M$ and the same Lipschitz constant $L$. This is fine from a mathematical point of view; from that of applications, it is something of interest to distinguish the bounds on each control.

\end{remark}

Let us start as usual by the discrete problem. 
Assuming that the initial datum $(X_0^{(N)},Y_0^{(N)})\in (L^p(\Om;\R^d))^N\times(\R^d)^M$, 
by Assumptions \ref{assumptions} and by Lemma \ref{lemma:finite.well.posedness}, we have the existence and uniqueness of a strong solution $(X^{(N)},Y^{(N)})\in (C([0,T];L^p(\Om;\R^d)))^{N+M}$, and moreover, the empirical measure $\mu_N(t)=\frac{1}{N}\sum_{n=1}^N\delta_{X_n^{(N)}(t)}\in L^p(\Om;\mW_p(\R^d))$ for all $t\in[0,T]$.

Thanks to this and to the Assumptions \ref{assumptions.control}, we get that the functional $\mF_N\colon\,E^M\longmapsto\R$ defined as
\begin{equation}\mF_N(h,g):=\mathbb{E}\Big[\int_0^T\Psi_\rho(h(t),g(Y^{(N)}(t),\mu_N(t)))\,dt + \int_0^T\Psi_\t(t,Y^{(N)}(t),\mu_N(t))\,dt + \Psi_\e(Y^{(N)}(T),\mu_N(T))\Big]
\end{equation}
is well-posed. We are interested in showing that there exists a solution to the optimal control problem
\begin{equation}\label{discrete.optimal.control.problem}\exists(h_{*,N},g_{*,N})\in \mM(\Om;E^M)\,\,\text{such that}\,\,\mF_N(h_{*,N},g_{*,N})=\min_{(h,g)\in\mM(\Om;E^M)}\mF_N(h,g).
\end{equation}
To this term, let us start by showing the stability of the constraints \eqref{SDE.finite}-\eqref{Initial.data.finite} with respect to variation of controls:
\begin{lemma}\label{lemma:stability.of.controls.finite}
In the Assumptions \ref{assumptions}, \ref{assumptions.control}, fixed any $p\ge2$ let us suppose that there exists a sequence $(h^{(j)},g^{(j)})_{j\in\N}\subset\mM(\Om;E^M)$ such that $\mathbb{P}-$almost surely,
\begin{align}&h^{(j)}\rightharpoonup h\quad\text{weakly in}\,\,L^1[0,T] \label{h.approximation.assumptions.finite},
\\&g^{(j)}\to g\quad\text{pointwise on}\,\,(\R^d)^M\times\mW_p(\R^d). \label{g.approximation.assumptions.finite}
\end{align}
Let $(X^{(N,j)},Y^{(N,j)})$ be the strong solution of \eqref{SDE.finite}-\eqref{Initial.data.finite} with initial datum \begin{equation*}(X_0^{(N)},Y_0^{(N)})\in (L^p(\Om;\R^d))^N\times(\R^d)^M\end{equation*}
and control function $u^{(j)}$ such that for all $m$ with $1\le m\le M$, $u_m^{(j)}=h_m^{(j)}\cdot g_m^{(j)}$. Furthermore, let $(X^{(N)},Y^{(N)})$ be the strong solution of \eqref{SDE.finite}-\eqref{Initial.data.finite} with initial datum 
\begin{equation*}(X_0^{(N)},Y_0^{(N)})\in(L^p(\Om;\R^d))^N\times(\R^d)^M\end{equation*}
and control function $u$ such that for all $m$ with $1\le m\le M$, $u_m=h_m\cdot g_m$.

\medskip

Then, calling $\mu_N^{(j)}$ the empirical measure associated to $(X^{(N,j)},Y^{(N,j)})$ and $\mu_N$ that of $(X^{(N)},Y^{(N)})$, it holds that
\begin{equation}\label{stability.of.real.solution.finite}\lim_{j\to\infty}\mathbb{E}\Big[\sup_{t\in[0,T]}\mW_1(\mu_N(t),\mu_N^{(j)}(t)) + \sup_{t\in[0,T]}|Y^{(N)}(t)-Y^{(N,j)}(t)|\Big]=0.
\end{equation}
\end{lemma}

\begin{proof}
By definitions of strong solution, arguing as done in Lemma \ref{lemma:finite.well.posedness} one has that on one hand
\begin{align}|X_n^{(N)}(t)-X_n^{(N,j)}(t)|^p
&\lesssim_{p,L}\int_0^t\Big\{\sup_{z\in[0,s]}|X^{(N)}(z) - X^{(N,j)}|^p + \sup_{z\in[0,s]}|Y^{(N)}(z)-Y^{(N,j)}(z)|^p\Big\}\,ds \notag
\\&+\sup_{s\in[0,t]}\Big|\int_0^s[\s_*^{(i)}(z,Y^{(N)},X_n^{(N)},\mu_N) - \s_*^{(i)}(z,Y^{(N,j)},X_n^{(N,j)},\mu_N^{(j)})]\,dW_n^{(i)}(z) \Big|^p \notag
\\&+\sup_{s\in[0,t]}\Big|\int_0^s[\s_*^{(c)}(z,Y^{(N)},X_n^{(N)},\mu_N) - \s_*^{(c)}(z,Y^{(N,j)},X_n^{(N,j)},\mu_N^{(j)})]\,dW^{(c)}(z) \Big|^p 
\label{1}
\end{align}
and on the other hand
\begin{align}|Y_m^{(N)}(t) - Y_m^{(N,j)}(t)|^p
&\lesssim_{p,L}\int_0^t\Big\{\sup_{z\in[0,s]}|X^{(N)}(z) - X^{(N,j)}(z)|^p + \sup_{z\in[0,s]}|Y^{(N)}(z)-Y^{(N,j)}(z)|^p\Big\}\,ds \notag
\\&+\Big|\underbrace{\int_0^t h_mg_m(Y^{(N)},\mu_N) - h_m^{(j)}g_m^{(j)}(Y^{(N,j)},\mu_N^{(j)})\,ds}_{=:(I)}\Big|^p.  \label{2}
\end{align}
In \eqref{2}, let us estimate $(I)$ by using \eqref{h.approximation.assumptions.finite}:
\begin{align*}(I)
&\le\Big|\int_0^t[h_m-h_m^{(j)}]\cdot g_m(Y^{(N)},\mu_N)\,ds\Big|
\\&+ \int_0^t|h_m^{(j)}|\cdot|g_m(Y^{(N)},\mu_N)-g_m^{(j)}(Y^{(N)},\mu_N)|\,ds
\\&+ \underbrace{\int_0^t|h_m^{(j)}|\cdot|g_m^{(j)}(Y^{(N)},\mu_N)-g_m^{(j)}(Y^{(N,j)},\mu_N^{(j)})|\,ds}_{=:(T)}
\\&\lesssim_{M'}\Big|\int_0^t[h_m-h_m^{(j)}]\cdot g_m(Y^{(N)},\mu_N)\,ds\Big|+\|g_m(Y^{(N)},\mu_N)-g_m^{(j)}(Y^{(N)},\mu_N)\|_{L^1[0,t]} + (T)
\\&\lesssim_{M'}\sum_{m=1}^M\sup_{s\in[0,t]}\Big[\Big|\int_0^s[h_m-h_m^{(j)}]\cdot g_m(Y^{(N)},\mu_N)\,dz\Big|+ \|g_m(Y^{(N)},\mu_N)-g_m^{(j)}(Y^{(N)},\mu_N)\|_{L^1[0,s]}\Big]
\\&+ (T)
\\&=:\sum_{m=1}^M R_m^{(j)}(t) + (T).
\end{align*}
In the first inequality, we observe that for all $t\in[0,T]$ and for all $1\le m\le M$, the first integral converges to $0$ as $j\to\infty$ because of \eqref{G.description} and \eqref{h.approximation.assumptions.finite}; the second integral because of \eqref{U.G.compatti}, \eqref{g.approximation.assumptions.finite} and \eqref{G.description} which allows to apply Lebesgue Dominated Convergence Theorem. Then, $R_m^{(j)}\to0$ as $j\to\infty$, for all $m$ with $1\le m\le M$ and for any $\om\in\Om$. As for $(T)$, it can be controlled by \eqref{h.approximation.assumptions.finite}, \eqref{G.description} and \eqref{Wasserstein.empirical.measure} as follows:
\begin{align}(T)
&\lesssim_{M',L}\int_0^t[\mW_1(\mu_N(s),\mu_N^{(j)}(s)) + |Y^{(N)} - Y^{(N,j)}|]\,ds \notag
\\&\lesssim_{M',L}\int_0^t[|X^{(N)}(s)-X^{(N,j)}(s)| + |Y^{(N)} - Y^{(N,j)}|]\,ds \label{T.estimate}
\end{align}
Coming back to \eqref{1} and \eqref{2}, if we call \begin{align*}&G(t):=\mathbb{E}\Big[\sup_{s\in[0,t]}|X^{(N)}(s)-X^{(N,j)}(s)|^p + \sup_{s\in[0,t]}|Y^{(N)}(s) - Y^{(N,j)}(s)|^p\Big],
\\&R^{(j)}:=\mathbb{E}\Big[\Big|\sum_{m=1}^M R_m^{(j)}(T)\Big|^p\Big],
\end{align*}
as usual we get
\begin{align*}G(t)\lesssim_{p,L,T,M'}\int_0^t G(s)\,ds + R^{(j)}.
\end{align*}
By \eqref{U.G.compatti} and \eqref{G.description}, we find out that $R_m^{(j)}(T)$ is uniformly bounded in $\Om$ with respect to $j$, so that by Lebesgue Dominated Convergence Theorem, $R^{(j)}\to0$ as $j\to\infty$: we conclude by Proposition \ref{prop:Gronwall.inequality} and by \eqref{Wasserstein.empirical.measure}.

\end{proof}

\begin{remark}\label{remark:extension.stability.lemma} One can extend Lemma \ref{lemma:stability.of.controls.finite} even when $j=N$, in other words, when also the number of herd individuals increases. Indeed, in the first inequality when estimating $(I)$, by uniform boundedness of $g^{(N)}$ one gets that the first integral converges to $0$; as for the second integral, one observes that $h_m^{(N)}$ is uniformly bounded and taken $Y$ the solution to the mean-field system \eqref{SDE.mean.field}-\eqref{Initial.data.mean.field} with control $u$, we also have
\begin{align*}|g_m(Y^{(N)},\mu_N) - g_m^{(N)}(Y^{(N)},\mu_N)| 
&\le\underbrace{|g_m(Y^{(N)},\mu_N) - g_m(Y,\mu)|}_{=:(I)}
\\&+\underbrace{|g_m(Y,\mu) - g_m^{(N)}(Y,\mu)|}_{=:(II)}
\\&+\underbrace{|g_m^{(N)}(Y,\mu) - g_m^{(N)}(Y^{(N)},\mu_N)|}_{=:(III)};
\end{align*}
the terms $(I),(III)$ tends to $0$ because by \eqref{G.description}, each $g_m$ is $L-$Lipschitz continuous and Theorem \ref{thm:propagation.of.chaos} holds, while the term $(II)$ converges to $0$ by \eqref{g.approximation.assumptions.finite}. All the integral converges then to $0$ by Lebesgue Dominated Convergence Theorem. This fact will be useful for the proof of Theorem \ref{thm:Gamma.convergence}.
    
\end{remark}

Thanks to this lemma, we can get the existence of minima for \eqref{F.N.finite}:
\begin{lemma}\label{lemma:F.N.has.solution}
In the Assumptions \ref{assumptions} and \ref{assumptions.control}, fixed any $p\ge2$ and supposing that $(X_0^{(N)},Y_0^{(N)})\in (L^p(\Om;\R^d))^N\times(\R^d)^M$ is the initial datum for \eqref{SDE.finite}-\eqref{Initial.data.finite}, we have that the problem \eqref{discrete.optimal.control.problem} has a solution $(h_{*,N},g_{*,N})\in\mM(\Om;E^M)$.
\end{lemma}

\begin{proof}
In this proof, we will refer to the notations introduced in the proof of Lemma \ref{lemma:stability.of.controls.finite}.
Let us firstly observe that by \eqref{U.G.compatti}, \eqref{Psi.in.assumptions} and \eqref{L.assumptions}, $\mF_N$ is bounded (from below), so that there exists a minimizing sequence $(h^{(j)},g^{(j)})_{j\in\N}\subset\mM(\Om;E^M)$ for $\mF_N$, namely, 
\begin{equation}\mF_N(h^{(j)},g^{(j)})\to\inf_{\mM(\Om;E^M)}\mF_N
\end{equation}
as $j\to\infty$.
We want to show that $\displaystyle\inf_{\mM(\Om;E^M)}\mF_N$ is realized by some $(h_{*,N},g_{*,N})\in \mM(\Om;E^M)$, which would be then a minimum for $\mF_N$. 

We need to guess it. On one hand, by \eqref{U.G.compatti}, $\mU$ is compact in $\R^{\ell\times d}$ so that we can suppose with no loss of generality that $\mathbb{P}-$almost surely,
\begin{equation}\label{hj.weak.conv.h.finite}h^{(j)}\rightharpoonup h_{*,N}\qquad\text{weakly in}\,\,L^1\end{equation}
with $h_{*,N}\in\mM(\Om;(L^1([0,T];\mU))^M)$.

On the other hand, since by Lemma \ref{lemma:finite.well.posedness} we have that $\mathbb{P}-$almost surely, $\mu_N^{(j)}(t)\in\mW_p(\R^d)$ for all $t\in[0,T]$, then we are allowed to restrict $\mathbb{P}^{(c)}-$almost surely $g^{(j)}$ to $\mW_p(\R^d)$, which is $\s-$compact and dense in $\mW_1(\R^d)$: by Ascoli-Arzelà, with no loss of generality we can suppose that there exists some $g_{*,N}\in\mG^M$ such that 
\begin{align}g^{(j)}\to g_{*,N}\qquad&\text{pointwise in}\,\,(\R^d)^M\times\mW_p(\R^d), \notag
\\&\text{uniformly in compact subsets of}\,\,(\R^d)^M\times\mW_p(\R^d). \label{gj.weak.con.g.finite}
\end{align}
If we then show that
\begin{equation}\label{final.condition.finite}\inf_{\mM(\Om;E^M)}\mF_N=\liminf_{j\to\infty}\mF_N(h^{(j)},g^{(j)})\ge\mF_N(h_{*,N},g_{*,N}), 
\end{equation}
then we are done.
We want to estimate $\mF_N$ from below. Recalling the expression \eqref{F.N.finite}, let us start by the endpoint cost. Since $\Psi_\e$ satisfies \eqref{e.assumptions}, there exists a concave continuity modulus $\om_1$, and so we get the estimate
\begin{align*}|\Psi_\e(Y^{(N)}(T),\mu_N(T))& - \Psi_\e(Y^{(N,j)}(T),\mu_N^{(j)}(T))|
\\&\le\om_1\Big(\sup_{t\in[0,T]}\mW_1(\mu_N(t),\mu_N^{(j)}(t)) + \sup_{t\in[0,T]}|Y^{(N)}(t) - Y^{(N,j)}(t)|\Big),
\end{align*}
from which we get by Jensen Inequality
\begin{align*}\mathbb{E}[|\Psi_\e(Y^{(N)}(T),\mu_N(T))& - \Psi_\e(Y^{(N,j)}(T),\mu_N^{(j)}(T))|]
\\&\le\om_1\Big(\mathbb{E}\Big[\sup_{t\in[0,T]}\mW_1(\mu_N(t),\mu_N^{(j)}(t)) + \sup_{t\in[0,T]}|Y^{(N)}(t) - Y^{(N,j)}(t)|\Big]\Big),
\end{align*}
which tends to $0$ as $j\to\infty$ by \eqref{hj.weak.conv.h.finite}, \eqref{gj.weak.con.g.finite} and \eqref{stability.of.real.solution.finite}.

Now, let us turn our attention to the transient cost. Since $\Psi_\t$ is uniformly continuous, there exists a concave continuity modulus $\om_2$, for which
\begin{align*}\int_0^T&|\Psi_\t(t,Y^{(N)},\mu_N) -\Psi_\t(t,Y^{(N,j)},\mu_N^{(j)})|\,dt
\\&\lesssim_T \om_2\Big(\sup_{t\in[0,T]}\mW_1(\mu_N(t),\mu_N^{(j)}(t)) + \sup_{t\in[0,T]}|Y^{(N)}(t) - Y^{(N,j)}(t)|\Big),
\end{align*}
and so by Jensen Inequality and again, the transient cost converges to $0$ as $j\to\infty$. 

Let us finally estimate the second integral in \eqref{F.N.finite}. 
Let us obseve that
\begin{align*}\int_0^T[\Psi_\rho(h_{*,N},g_{*,N}(Y^{(N)},\mu_N)) -& \Psi_\rho(h^{(j)}, g^{(j)}(Y^{(N,j)},\mu_N^{(j)}))]\,dt
\\&=\underbrace{\int_0^T[\Psi_\rho(h_{*,N},g_{*,N}(Y^{(N)},\mu_N)) - \Psi_\rho(h^{(j)}, g_{*,N}(Y^{(N)},\mu_N))]\,dt}_{=:(I)}
\\&\,+\underbrace{\int_0^T[\Psi_\rho(h^{(j)}, g_{*,N}(Y^{(N)},\mu_N)) - \Psi_\rho(h^{(j)}, g^{(j)}(Y^{(N)},\mu_N))]\,dt}_{=:(II)}
\\&\,+\underbrace{\int_0^T[\Psi_\rho(h^{(j)}, g^{(j)}(Y^{(N)},\mu_N)) - \Psi_\rho(h^{(j)}, g^{(j)}(Y^{(N,j)},\mu_N^{(j)}))]\,dt}_{=:(III)}.
\end{align*}
Let us start by the term $(III)$.
Since $(h^{(j)},g^{(j)})\in\mM(\Om;E^M)$, by \eqref{U.G.compatti} we get that $\Psi(h^{(j)}(t),\cdot)$ gets uniformly continuous $\mathbb{P}-$almost surely when acting on the second variable, so it has a concave continuity modulus $\om_3$ and in the same way as done above, we get the estimate
\begin{align*}\mathbb{E}[|(III)|]\lesssim_ {T,L}\om_3\Big(\mathbb{E}\Big[\sup_{t\in[0,T]}\mW_1(\mu_N(t),\mu_N^{(j)}(t)) + \sup_{t\in[0,T]}|Y^{(N)}-Y^{(N,j)}|\Big]\Big),
\end{align*}
which converges to $0$ as $j\to\infty$ like the previous terms. 

As for $(II)$, since \eqref{gj.weak.con.g.finite} holds, then by analogous considerations done above for $\Psi$, one has that since
\begin{equation*}|\Psi_\rho(h^{(j)},g_{*,N}(Y^{(N)},\mu_N)) - \Psi_\rho(h^{(j)},g^{(j)}(Y^{(N)},\mu_N)|\le\om_3(|g_{*,N}(Y^{(N)},\mu_N) - g^{(j)}(Y^{(N)},\mu_N)|)\to0,
\end{equation*}
as $j\to\infty$, for all $t\in[0,T]$ and almost surely in $\Om$; furthermore, since \eqref{U.G.compatti} and \eqref{Psi.in.assumptions} hold, then applying Lebesgue Dominated Convergence Theorem one also gets that
$\mathbb{E}[(II)]\to0$
as $j\to\infty$.

Finally, as for the term $(I)$, by \eqref{Psi.in.assumptions} we have that the function $h\longmapsto\Psi_\rho(h,g_{*,N}(Y(t),\nu_t))$ is convex, and as such it is lower semicontinuous with respect to the weak convergence in $L^1([0,T];\R^{d\times\ell})$, so that
\begin{equation*}\liminf_{j\to\infty}\int_0^T\Psi_\rho(h^{(j)},g_{*,N}(Y^{(N)},\mu_N))\,dt\ge\int_0^T\Psi_\rho(h_{*,N},g_{*,N}(Y^{(N)},\mu_N))\,dt:
\end{equation*}
since \eqref{U.G.compatti} and \eqref{Psi.in.assumptions} hold, then $\Psi_\rho(h^{(j)}(t),g_{*,N}(Y^{(N)}(t),\mu_N(t)))$ is bounded, so that by Fatou Lemma we get \eqref{final.condition.finite}.
    
\end{proof}

Now, let us focus on the mean-field problem. Let us suppose that the initial datum $(X_0,Y_0)\in L^p(\Om;\R^d)\times(\R^d)^M$ for $p\ge 2$ so that by Lemma \ref{thm:mean.field.well.posedness}, the problem \eqref{SDE.mean.field}-\eqref{Initial.data.mean.field} holds a solution $(X,Y)\in C([0,T];L^p(\Om;\R^d)\times L^{p}(\Om;(\R^d)^M))$ and moreover, the associated conditional measure given $\mF^{(c)}$ is such that $\mu^{(i)}\in C([0,T];L^p(\Om;\mW_p(\R^d)))$ (see boundedness estimates \eqref{mean.field.solution.boundedness}).

Even here, thanks to the Assumptions \ref{assumptions.control} we get that the functional $\mF\colon\,E^M\longmapsto\R$ defined as
\begin{equation*}\mF(h,g):=\mathbb{E}\Big[\int_0^T\Psi_\rho(h(t),g(Y(t),\mu(t)))\,dt + \int_0^T\Psi_\t(t,Y(t),\mu(t))\,dt + \Psi_\e(Y(T),\mu(T))\Big]
\end{equation*}
is well-posed. We are interested in solving the optimal control problem
\begin{equation}\label{mean.field.optimal.control.problem}\exists(h_*,g_*)\in \mM(\Om;E^M)\,\,\text{such that}\,\,\mF(h_*,g_*)=\min_{(h,g)\in\mM(\Om;E^M)}\mF(h,g).
\end{equation}
Also here, we provide a lemma analogous to Lemma \ref{lemma:stability.of.controls.finite}:

\begin{lemma}\label{lemma:stability.of.controls.mean.field}
In the Assumptions \ref{assumptions}, \ref{assumptions.control}, for any fixed $p\ge2$ let us suppose that there exists a sequence $(h^{(j)},g^{(j)})_{j\in\N}\subset\mM(\Om;E^M)$ such that $\mathbb{P}-$almost surely,
\begin{align}&h^{(j)}\rightharpoonup h\quad\text{weakly in}\,\,L^1 \label{h.approximation.assumptions.mean.field},
\\&g^{(j)}\to g\quad\text{strongly on}\,\,C((\R^d)^M\times\mW_p(\R^d)).\label{g.approximation.assumptions.mean.field}
\end{align}
Let $(X^{(j)},Y^{(j)})$ be the strong solution of \eqref{SDE.mean.field}-\eqref{Initial.data.mean.field} with initial datum 
\begin{equation*}(X^{(0)},Y^{(0)})\in L^p(\Om;\R^d)\times(\R^d)^M\end{equation*}
and control function $u^{(j)}$ such that for all $m$ with $1\le m\le M$, $u_m^{(j)}=h_m^{(j)}\cdot g_m^{(j)}$. Furthermore, let $(X,Y)$ be the strong solution of \eqref{SDE.mean.field}-\eqref{Initial.data.mean.field} with the same initial datum $(X^{(0)},Y^{(0)})$ and control function $u$ such that for all $m$ with $1\le m\le M$, $u_m=h_m\cdot g_m$.

\medskip

Then, calling $\mu_t^{(i,j)}$ the conditional measure given $\mF_t^{(c)}$ associated to $(X^{(j)}(t),Y^{(j)}(t))$ and calling $\mu^{(i)}(t)$ that of $(X(t),Y(t))$, it holds that
\begin{equation}\label{stability.of.real.solution}\lim_{j\to\infty}\mathbb{E}\Big[\sup_{t\in[0,T]}\mW_1(\mu^{(i)}(t),\mu^{(i,j)}(t)) + \sup_{t\in[0,T]}|Y(t)-Y^{(j)}(t)|\Big]=0.
\end{equation}

\end{lemma}

\begin{proof}
By definition of strong solutions, arguing as done in Lemma \ref{thm:mean.field.well.posedness} we have
\begin{align*}|X(t) - X^{(j)}(t)|^p
&\lesssim_{p,L,T}\int_0^t\sup_{z\in[0,s]}\Big\{|X(z)-X^{(j)}(z)|^p + |Y(z) - Y^{(j)}(z)|^p \Big\}\,ds
\\&+\sup_{s\in[0,t]}\Big|\int_0^s[\s_*^{(i)}(z,Y,X,\mu^{(i)}) - \s_*^{(i)}(z,Y^{(j)},X^{(j)},\mu^{(i,j)})]\,dW^{(i)}(z)\Big|^p
\\&+\sup_{s\in[0,t]}\Big|\int_0^s[\s_*^{(c)}(z,Y,X,\mu^{(i)}) - \s_*^{(c)}(z,Y^{(j)},X^{(j)},\mu^{(i,j)})]\,dW^{(c)}(z)\Big|^p
\end{align*}
as well as
\begin{align*}|Y_m(t) - Y_m^{(j)}(t)|^p
&\lesssim_{p,L,T}\int_0^t\sup_{z\in[0,s]}\Big\{|X(z)-X^{(j)}(z)|^p + |Y(z) - Y^{(j)}(z)|^p \Big\}\,ds
\\& + \Big|\int_0^t h_mg_m(Y,\mu) - h_m^{(j)}g_m^{(j)}(Y^{(j)},\mu^{(j)})\,ds\Big|^p
\end{align*}
Then, by using Assumptions \ref{assumptions.control}, one can reason exactly as done in Lemma \ref{lemma:stability.of.controls.finite} to conclude.
    
\end{proof}

Thanks to this lemma, we can get the existence of minima for $\mF$ as well:
\begin{lemma}\label{lemma:F.has.solution}
In the Assumptions \ref{assumptions} and \ref{assumptions.control}, supposing that $p\ge2$ and $(X_0,Y_0)\in L^p(\Om;\R^d)\times(\R^d)^M$ is the initial datum for \eqref{SDE.mean.field}-\eqref{Initial.data.mean.field}, we have that the problem \eqref{mean.field.optimal.control.problem} has a solution $(h_*,g_*)\in\mM(\Om;E^M)$.
\end{lemma}

\begin{proof} The proof can be done by following the same arguments done for proving Lemma \ref{lemma:F.N.has.solution}.

\end{proof}

Having shown the well-posedness for the optimal control problems \eqref{F.N.finite} and \eqref{F.mean.field} respectively for the discrete problem \eqref{SDE.finite}-\eqref{Initial.data.finite} and the mean-field one \eqref{SDE.mean.field}-\eqref{Initial.data.mean.field}, now we want to show that also in these terms the mean-field control problem is an approximation for the discrete ones. More precisely, we want to show the minima of \eqref{F.N.finite} $\Gamma-$converge to those of \eqref{F.mean.field}.

\begin{theorem}[$\Gamma-$convergence]\label{thm:Gamma.convergence}
The following facts hold.

$(i)$ In the notations and assumptions of Theorem \ref{thm:propagation.of.chaos} and Lemmas \ref{lemma:F.N.has.solution} and \ref{lemma:F.has.solution}, equipping $E$ with the product topology between the weak one for $L^1$ and the strong one for $C((\R^d)^M\times\mW_1(\R^d);\R^d))$, we get that
\begin{equation}\label{G.convergence}\mF_N\stackrel{\Gamma}{\to}\mF.
\end{equation}

$(ii)$ In the notations and assumptions made above, one has
\begin{equation}\label{convergence.of.minima}\lim_{N\to\infty}\min_{(h,g)\in\mM(\Om;E^M)}\mF_N(h,g) = \min_{(h,g)\in \mM(\Om;E^M)}\mF(h,g).
\end{equation}
    
\end{theorem}

\begin{proof}
$(i)$ Let us consider a generic sequence $(h_N,g_N)_{N\in\N}$ converging $\mathbb{P}-$almost surely in $\mM(\Om;E^M)$ to some point $(h,g)\in\mM(\Om;E^M)$ with respect to the topology given in the statement.
We will use the following notations for any $f\in\{X,Y,\mu\}$: by simply $f$ we are referring to equations \eqref{SDE.mean.field}-\eqref{Initial.data.mean.field} with control $u=(u_m:=h_mg_m)_{1\le m\le M}$; by $f_N$ we are referring to the latter equations with control $u_N=((u_N)_m:=(h_N)_m(g_N)_m)_{1\le m\le M}$; by $f^{(N)}$ we are referring to the equations \eqref{SDE.finite}-\eqref{Initial.data.finite} with control $u$; finally, by $f_N^{(N)}$ we are referring to the latter equations with control $u_N$.

\medskip

To establish the $\Gamma-$convergence of $\mF_N$ to $\mF$, we must show the liminf inequality first. Let us write
\begin{align*}\mF(h,g) - \mF_N(h_N,g_N)
&=\underbrace{\mathbb{E}\Big[\int_0^T[\Psi_\rho(h,g(Y,\mu^{(i)})) - \Psi_\rho(h_N,g_N(Y_N^{(N)},\mu_N^{(N)}))]\,dt\Big]}_{=:(I)}
\\&+\underbrace{\mathbb{E}\Big[\int_0^T[\Psi_\t(t,Y,\mu^{(i)}) - \Psi_\t(t,Y_N^{(N)},\mu_N^{(N)})]\,dt\Big]}_{=:(II)}
\\&+\underbrace{\mathbb{E}[\Psi_\e(Y(T),\mu^{(i)}(T)) - \Psi_\e(Y^{(N)}(T),\mu_N^{(N)}(T))]}_{=:(III)}
;
\end{align*}
to estimate $(II)$, one argues as in Lemmas \ref{lemma:F.N.has.solution} and \ref{lemma:F.has.solution} to get that
\begin{align*}|(II)|
&\lesssim_ {T,L}\om_2\Big(\mathbb{E}\Big[\sup_{t\in[0,T]}\mW_1(\mu^{(i)}(t),\mu_N^{(N)}(t)) + \sup_{t\in[0,T]}|Y(t)-Y_N^{(N)}(t)| \Big] \Big)
\\&\lesssim_{T,L}\om_2\Big(\mathbb{E}\Big[\sup_{t\in[0,T]}\mW_1(\mu^{(i)}(t),\mu^{(N)}(t)) + \sup_{t\in[0,T]}|Y(t)- Y^{(N)}(t)| \Big]
\\&\qquad\qquad+\mathbb{E}\Big[\sup_{t\in[0,T]}\mW_1(\mu^{(N)}(t),\mu_N^{(N)}(t)) + \sup_{t\in[0,T]}|Y^{(N)}(t) - Y_N^{(N)}(t)|\Big]\Big)
,
\end{align*}
which converges to $0$ as $N\to\infty$ thanks to Theorem \ref{thm:propagation.of.chaos} and Lemma \ref{lemma:stability.of.controls.mean.field}, Remark \ref{remark:extension.stability.lemma}. Similarly, one estimates $(III)$.

As for the term $(I)$, we write
\begin{align*}(I)
&=\underbrace{\mathbb{E}\Big[\int_0^T[\Psi_\rho(h,g(Y,\mu^{(i)})) - \Psi_\rho(h_N,g(Y,\mu^{(i)}))]\,dt \Big]}_{=:(I.1)}
\\&+\underbrace{\mathbb{E}\Big[\int_0^T[\Psi_\rho(h_N,g(Y,\mu^{(i)})) - \Psi_\rho(h_N,g(Y_N^{(N)},\mu_N^{(N)}))]\,dt \Big]}_{=:(I.2)}.
\end{align*}
Let us have a look at the term $(I.2)$. By \eqref{U.G.compatti} and \eqref{Psi.in.assumptions}, the integrand function is uniformly continuous, so that for it there exists a concave continuity modulus $\om_3$. We then get the estimate
\begin{align*}|(I.2)|
&\lesssim_T\om_3\Big(\mathbb{E}\Big[\sup_{t\in[0,T]}|g(Y(t),\mu^{(i)}(t)) - g(Y_N^{(N)}(t),\mu_N^{(N)}(t))|\Big]\Big),
\end{align*}
which tends to $0$ by \eqref{U.G.compatti} and Lebesgue Dominated Convergence Theorem. As for $(I.1)$, by convexity of $\Psi$ with respect to the first variable and consequently by its lower semicontinuity with respect to the weak $L^1$ topology, one can apply Fatou Lemma to get that $(II.1)\le0$: as a consequence, when $N\to\infty$, plugging all our estimates together we get the desired liminf-inequality.

Now, one has to find a recovering sequence to have the $\Gamma-$convergence: it is enough to take $(h_N,g_N):=(h,g)$, and so we are done. 

\medskip

$(ii)$ By Lemmas \ref{lemma:F.N.has.solution} and \ref{lemma:F.has.solution}, the statement is well-posed.
Let then $((h_{N,*},g_{N,*}))_{N\in\N}$ be a sequence of solutions to the optimal control problems \eqref{discrete.optimal.control.problem}. By the same considerations done in Lemma \ref{lemma:F.N.has.solution}, one can suppose that up to subsequences, there exists $(h_*,g_*)\in \mM(\Om;E^M)$ such that $\mathbb{P}-$almost surely, $h_{N,*}\rightharpoonup h_*$ weakly in $L^1$ and $g_{N,*}\to g_*$ in $C(K;\R^\ell)$ for any bounded set $K\subset(\R^d)^M\times\mW_p(\R^d)$ (which is relatively compact in $(\R^d)^M\times\mW_1(\R^d)$); by boundedness estimate \eqref{mean.field.solution.boundedness}, one can restrict each $g_{N,*}$ as well as $g_*$ can be restricted to such a $K$.
By what proved in point $(i)$, one has that $\displaystyle\lim_{N\to\infty}\mF_N(h,g)=\mF(h,g)$ for all $(h,g)\in E^M$, so that 
\begin{align*}\mF(h_*,g_*)
&\le\liminf_{N\to\infty}\mF_N(h_{N,*},g_{N,*})
\\&\le\limsup_{N2\to\infty}\mF_N(h_{N,*},g_{N,*})
\\&\le\lim_{N\to\infty}\mF_N(h,g)
\\&=\mF(h,g),
\end{align*}
so that we conclude by arbitrariness of $(h,g)\in \mM(\Om;E^M)$.

\end{proof}

\newpage

\begin{flushright}

\textbf{Giuseppe La Scala}

Mathematical and Physical Sciences for Advanced Materials and Technologies

Scuola Superiore Meridionale

Via Mezzocannone, 4, 80138 Naples, Italy

giuseppe.lascala-ssm@unina.it

\end{flushright}

\end{document}